\documentclass{siamltex}

\usepackage[utf8]{inputenc}
\usepackage{graphicx}
\usepackage{amsmath}
\usepackage{amssymb}
\usepackage{mathtools}
\usepackage{enumerate}
\usepackage{color}
\usepackage{epsfig}
\usepackage{tikz}
\usepackage[ruled,vlined]{algorithm2e}

\newcommand{\M}{\mathcal{M}}

\newcommand{\R}{\mathbb{R}}

\newcommand{\eps}{\varepsilon}

\newcommand{\A}{{\bf A}}
\newcommand{\B}{{\bf B}}

\renewcommand{\S}{{\bf S}}

\newcommand{\bcl}{\color{black}}

\newcommand{\ecl}{\color{black}}

\DeclareMathOperator{\F}{\bf{F}}

\DeclareMathOperator{\U}{\bf{U}}
\DeclareMathOperator{\V}{\bf{V}}
\DeclareMathOperator{\W}{\bf{W}}
\DeclareMathOperator{\Q}{\bf{Q}}
\DeclareMathOperator{\I}{\bf{I}}
\DeclareMathOperator{\K}{\bf{K}}

\DeclareMathOperator{\ten}{Ten}
\DeclareMathOperator{\mat}{\bf{Mat}}
\DeclareMathOperator{\bigtimes}{{\hbox{\large\sf X}}}

\usepackage{accents}

\allowdisplaybreaks

\title{Time integration of  tree tensor networks}
\author{Gianluca Ceruti\footnotemark[1] \and Christian Lubich\footnotemark[1] \and Hanna Walach\footnotemark[1]}
\date{}%\today} 

\begin{document}

\maketitle

\renewcommand{\thefootnote}{\fnsymbol{footnote}}
\footnotetext[1]{Mathematisches Institut,
       Universit\"at T\"ubingen,
       Auf der Morgenstelle 10,
       D--72076 T\"ubingen,
       Germany. Email: {\tt \{ceruti,lubich,walach\}@na.uni-tuebingen.de}}

\begin{abstract}
Dynamical low-rank approximation by tree tensor networks is studied for the data-sparse approximation of large time-dependent data tensors and unknown solutions to tensor differential equations.  A time integration method for tree tensor networks of prescribed tree rank is presented and analyzed.  It extends the known projector-splitting integrators for dynamical low-rank approximation by matrices and rank-constrained Tucker tensors and is shown to inherit their favorable properties.  The integrator is based on recursively applying the low-rank Tucker tensor integrator. In every time step, the integrator climbs up and down the tree: it uses a recursion that passes from the root to the leaves of the tree for the construction of initial value problems on subtree tensor networks using appropriate restrictions and prolongations, and another recursion that passes from the leaves to the root for the update of the factors in the tree tensor network.
The integrator reproduces given time-dependent tree tensor networks of the specified tree rank exactly and is robust to the typical presence of small singular values in matricizations of the connection tensors, in contrast to standard integrators applied to the differential equations for the factors in the dynamical low-rank approximation by tree tensor networks. 
\end{abstract}

{
	\begin{keywords}
		Tree tensor network, tensor differential equation, dynamical low-rank approximation, time integrator 
	\end{keywords}
	\begin{AMS}
		15A69, 65L05, 65L20, 65L70
	\end{AMS}
	\pagestyle{myheadings}
	\thispagestyle{plain}
	\markboth{G.~CERUTI, CH.~LUBICH AND H.~WALACH}{TIME INTEGRATION OF TREE TENSOR NETWORKS}
}

\section{Introduction}

%
%
%\begin{equation} \label{eq:fullEq}
%	\dot{A}(t) = F(t, A(t)), \quad A(t_0) = A^0.
%\end{equation}
%
%\begin{equation} \label{eq:projEq}
%	\dot{Y}(t) = P(Y) F(t, Y(t)), \quad Y(t_0) = Y^0.
%\end{equation}

For the approximate solution of the initial value problem for a (huge) system of differential equations for the tensor $A(t)\in \R^{n_1\times\ldots\times n_d}$, 
\begin{equation} \label{eq:fullEq}
	\dot{A}(t) = F(t, A(t)), %\quad A(t_0) = A^0.
\end{equation}
we aim to construct $Y(t)\approx A(t)$ in an approximation manifold~$\mathcal{M}$ of much smaller dimension, which in the present work will be chosen as a manifold of tree tensor networks of fixed tree rank. This shall provide a data-sparse computational approach to high-dimensional problems that cannot be treated by direct time integration because of both excessive memory requirements and computational cost. 

A differential equation for $Y(t)\in \mathcal{M}$ is obtained by choosing the time derivative
$\dot Y(t)$ as that element in the tangent space $T_{Y(t)}\M$ for which
$$
\| \dot Y(t) - F(t,Y(t)) \| \quad\text{is minimal},
$$
where the norm is chosen as the Euclidean norm of the vector of the tensor entries. In the quantum physics and chemistry literature, this approach is known as the Dirac--Frenkel time-dependent variational principle, named after work by Dirac in 1930 who used the approach in the context of what is nowadays known as the time-dependent Hartree--Fock method for the multi-particle time-dependent Schr\"odinger equation; see, e.g., \cite{KrS81,Lu08}. Equivalently, this minimum-defect condition can be stated as a Galerkin condition on the state-dependent approximation space $T_{Y(t)}\M$,
$$
\dot Y(t) \in T_{Y(t)}\M \quad\text{such that}\quad \langle\dot Y(t) - F(t,Y(t)), Z \rangle=0 \quad\ \forall \, Z\in T_{Y(t)}\M.
$$
Using the orthogonal projection $P(Y)$ onto the tangent space $T_Y\M$, this can be reformulated as the (abstract) ordinary differential equation on $\M$,
\begin{equation} \label{eq:projEq}
	\dot{Y}(t) = P(Y(t)) F(t, Y(t)).    %, \quad Y(t_0) = Y^0.
\end{equation}
This equation needs to be solved numerically in an efficient and robust way. For fixed-rank matrix and tensor manifolds,
the orthogonal projection $P(Y)$ turns out to be an alternating sum of subprojections, which reflects the multilinear structure of the problem. The explicit form of the tangent space projection in the low-rank matrix case as an alternating sum of three subprojections was derived in \cite{KoL07} and was used  in \cite{LuO14} to derive a projector-splitting integrator for low-rank matrices, which efficiently updates an SVD-like low-rank factorization in every time step and which is robust to the typically arising small 
singular values that cause severe difficulties with standard integrators applied to the system of differential equations for the factors of the SVD-like decomposition of the low-rank matrices; see \cite{KiLW16}. The projector-splitting integrator was extended to tensor trains / matrix product states in \cite{LuOV15}; see also \cite{HaLOVV16} for a description of the algorithm in a physical idiom. 
The projector-splitting integrator was extended  to  Tucker tensors of fixed multilinear rank in \cite{Lu15}.
A reinterpretation was given in \cite{LuVW18}, in which the Tucker tensor integrator was rederived by recursively performing inexact substeps in the matrix projector-splitting integrator applied to matricizations of the tensor differential equation followed by retensorization. This interpretation made it possible to show that the Tucker integrator inherits the favorable robustness properties of the low-rank matrix projector-splitting integrator.

In the present paper we take up such a recursive approach to derive an integrator for general tree tensor networks, which is shown to be efficiently implementable (provided that the righthand side function $F$ can be efficiently evaluated on tree tensor networks in factorized form) and to inherit the robust convergence properties of the low-rank matrix, Tucker tensor and tensor train / matrix product state integrators shown previously in \cite{KiLW16,LuVW18}. The proposed integrator for tree tensor networks reduces to the well-proven projector-splitting integrators in the particular cases of Tucker tensors and tensor trains / matrix product states. We expect (but will not prove) that it can itself be interpreted as a projector-splitting integrator based on splitting the tangent space projection of the fixed-rank tree tensor network manifold.

\bcl
For a special tree, this integrator (given in an {\it ad hoc} formulation) was already used in \cite{LuE18} for the Vlasov--Poisson equation of plasma physics. In that case, the tree is given by the separation $((x_1,x_2,x_3),(v_1,v_2,v_3))$ of the position and velocity variables, which are further separated into their Cartesian coordinates. Very recently,  this tree tensor network integrator (or very similar versions) was  applied to  problems of current interest in quantum physics in \cite{BaA19} and \cite{KlLR20},
studying multi-orbital Anderson impurity models and the dynamics in two-dimensional quantum lattices, respectively.
None of these papers gives a systematic construction and numerical analysis of the integrator. This is the objective of the present paper.
\ecl

In Section \ref{sec:ttn} we introduce notation and the formulation of  tree tensor networks as  multilevel-structured Tucker tensors and give basic properties, emphasizing orthonormal factorizations. The tree tensor network (TTN) is constructed from the basis matrices at the leaves and the connection tensors at the inner vertices of the tree in a multilinear recursion that passes from the leaves to the root of the tree.

In Section~\ref{sec:eti} we recall the algorithm of the Tucker tensor integrator of \cite{LuVW18} and extend it to the case of several Tucker tensors with the same basis matrices. This extended Tucker integrator, which is nothing but the Tucker integrator for an extended Tucker tensor, is a basic building block of the integrator for tree tensor networks.

In Section~\ref{sec:ttni}, the main algorithmic section of this paper, we derive the recursive TTN integrator and discuss the basic algorithmic aspects. In every time step, the integrator uses a recursion that passes from the root to the leaves of the tree for the construction of initial value problems on subtree tensor networks using appropriate restrictions and prolongations, and another recursion that passes from the leaves to the root for the update of the factors in the tree tensor network. The integrator only solves low-dimensional matrix differential equations (of the dimension of the basis matrices at the leaves) and low-dimensional tensor differential equations (of the dimension of the connection tensors at the inner vertices of the tree), alternating with orthogonal matrix decompositions of such small matrices and of matricizations of the connection tensors.

In Section~\ref{sec:exact} we prove a remarkable exactness property: if $F(t,Y)=\dot A(t)$ for a given tree tensor network $A(t)$ of the specified tree rank, then the recursive TTN integrator for this tree rank reproduces $A(t)$ exactly. This exactness property is proved using the analogous exactness property of the Tucker tensor integrator proved in \cite{LuVW18}, which in turn was proved using the exactness property for the matrix projector-splitting integrator that was discovered and proved in \cite{LuO14}.

In Section~\ref{sec:err} we prove first-order error bounds that are independent of small singular values of matricizations of the connecting tensors. The proof relies on the similar error bound for the Tucker integrator \cite{LuVW18}, which in turn relies on such an error bound for the fixed-rank matrix projector-splitting integrator proved in \cite{KiLW16}, in a proof that uses in an essential way the exactness property.
The robustness to small singular values distinguishes the proposed integrator substantially from standard integrators applied to the differential equations for the basis matrices and connection tensors derived in \cite{WaT03}. We note that the proposed TTN integrator foregoes the formulation of these differential equations for the factors. The ill-conditioned density matrices whose inverses appear in these differential equations are never formed, let alone inverted, in the TTN integrator.
%We expect (but do not prove) that our recursive TTN integrator can be interpreted as a projector-splitting integrator for the differential equation \eqref{eq:projEq}, since it reduces to the projector-splitting integrators for low-rank matrices, Tucker tensors and tensor trains / matrix product states in these special cases.

The present paper thus completes a path to extend the low-rank matrix projector-splitting integrator of \cite{LuO14}, together with its favorable properties, from the dynamical low-rank approximation by matrices of a prescribed rank to Tucker tensors of prescribed multilinear rank to general tree tensor networks of prescribed tree rank.

In Section~\ref{sec:num} we present a numerical experiment which shows the error behaviour of the proposed integrator in accordance with the theory. We choose the example of retraction of the sum of a tree tensor network and a tangent network, which is an operation needed in many optimization algorithms for tree tensor networks; cf.~\cite{AbO15} for the low-rank matrix  case. The corresponding example was already chosen in numerical experiments for the low-rank matrix, tensor train and Tucker tensor cases in \cite{LuO14,LuOV15,LuVW18}, respectively. It is beyond the scope of this paper to present the results of numerical experiments with the recursive TTN integrator in actual applications of tree tensor networks in physics, chemistry or other sciences. We note, however, that striking numerical experiments with this integrator 
\bcl
have already been reported 
for the Vlasov--Poisson equation of plasma physics in \cite{LuE18} and for problems in quantum physics in \cite{BaA19,KlLR20}.

While we describe the TTN integrator for real tensors, the algorithm and its properties extend in a straightforward way to complex tensors as arise in quantum physics. Only some additional care in using transposes $\U^\top$ versus adjoints $\U^*=\overline \U^\top$ is required for this extension.

Throughout the paper, we denote tensors by roman capitals and matrices by boldface capitals.
%, vectors by boldface lowercase, and scalars by roman lowercase. 
%The transpose of a matrix $\V$ is denoted $\V^\top$. 
%The identity matrix in varying dimensions is denoted $\I$.  
\ecl
\bigskip
%\bigskip

\pagebreak[3]
\section{Preparation: Matrices, Tucker tensors, tree tensor networks, and their ranks}
\label{sec:ttn}

\subsection{Matrices of rank $r$}
The singular value decomposition (SVD) shows that a matrix $\A\in \R^{m\times n}$ is of rank $r$ if and only if it can be factorized as
$$
\A = \U \S \V^\top,
$$
where $\U\in \R^{m\times r}$ and $\V \in \R^{n\times r}$ have orthonormal columns, and $\S\in \R^{r\times r}$ has full rank $r$. 
\bcl The decomposition is not unique. (The SVD is a particular choice with diagonal $\S$.) \ecl
The real $m\times n$ matrices of rank (exactly) $r$ are known to form a smooth embedded manifold in $\R^{m\times n}$ \cite{HeM94}.

\subsection{Tucker tensors of multilinear rank $(r_i)$}
For a tensor $A\in \R^{n_1\times\dots \times n_d}$, the multilinear rank $(r_1,\dots,r_d)$ is defined as the $d$-tuple of the ranks $r_i$ of the matricizations $\mat_i(A)\in \R^{n_i\times n_{\neg i}}$ for $i=1,\dots,d$, where $n_{\neg i} =\prod_{j\ne i} n_j$. We recall that the $i$th matricization aligns in the $k$th row (for $k=1,\dots,n_i$) all entries of $A$ that have the index $k$ in the $i$th position, usually ordered co-lexicographically. The inverse operation is tensorization of the matrix, denoted by $\ten_i$:
$$
\mathbf{A}_{(i)} = \mat_i(A)\in \R^{n_i\times n_{\neg i}} \quad\text{ if and only if }\quad
A = \ten_i(\mathbf{A}_{(i)}) \in \R^{n_1\times\dots \times n_d}.
$$
It is known from \cite{DeLDV00} that the tensor $A$ has multilinear rank $(r_1,\dots,r_d)$ if and only if it can be factorized as a {\it Tucker tensor} (we adopt the shorthand notation from~\cite{KoB09})
\bcl
\begin{align}\label{tucker}
&A = C \times_1 \U_1 \times_2 \U_2 \dots \times_d \U_d = C  \bigtimes_{i=1}^d \U_i, \qquad \\
\text{i.e.,} &\quad a_{k_1,\dots,k_d} = \sum_{l_1=1}^{r_1}\dots \sum_{l_d=1}^{r_d}
c_{l_1,\dots,l_d} u_{k_1,l_1}\dots u_{k_d,l_d},
\nonumber
\end{align}
\ecl
where the \emph{basis matrices} $\U_i\in \R^{n_i\times r_i}$ have orthonormal columns and the \emph{core tensor} $C\in \R^{r_1\times\dots \times r_d}$ has full multilinear rank $(r_1,\dots,r_d)$. (This requires a compatibility condition among the ranks: $r_i \le \prod_{j\ne i} r_j$. In particular, this condition is satisfied if all ranks $r_i$ are equal.) \bcl As in the matrix case ($d=2$), this decomposition is not unique.\ecl

A useful formula for the matricization of Tucker tensors is
\begin{equation}\label{unfolding}
\mat_i \bigl( C  \bigtimes_{j=1}^d \U_j \bigr) = \U_i \mat_i(C) \biggl(\bigotimes_{j\ne i} \U_j^\top \biggr),
\end{equation}
where $\otimes$ denotes the Kronecker product of matrices.

The tensors of given dimensions $(n_1,\dots,n_d)$ and fixed multilinear rank $(r_1,\dots,r_d)$ are known to form a smooth embedded manifold in $\R^{n_1\times\dots \times n_d}$.
% insert citation  for smooth manifold.

\subsection{Orthonormal tree tensor networks of tree rank $(r_\tau)$} \label{subsec:ttn}
A tree tensor network is a multilevel-structured Tucker tensor where the  configuration is described by a  tree. The notion of a `tree tensor network' was coined in the quantum physics literature \cite{ShDV06}, but tree tensor networks were actually already used a few years earlier in the multilayer MCTDH method of chemical physics \cite{WaT03}. In the mathematical literature, tree tensor networks with binary trees have been studied as `hierarchical tensors' \cite{Ha12} and with general trees as `tensors in tree-based tensor format' \cite{FaHN15,FaHN18}. We remark that Tucker tensors and  matrix product states / tensor trains \cite{PeVWC07,Os11} are particular instances of tree tensor networks, whose trees are trees of minimal height (bushes) and  binary trees of maximal height, respectively. 

\bcl
In an informal way, one arrives at a tree tensor network by first considering a tensor in Tucker format, in which then the basis matrices are tensorized and approximated by tensors in a low-rank Tucker format. Their basis matrices are again tensorized and approximated by tensors in low-rank Tucker format, and so on over multiple levels. A different viewpoint common in quantum physics is to describe a tree tensor network as resulting from  a collection of tensors that have two types of indices:
those corresponding to physical degrees of freedom and further auxiliary indices that always appear on two tensors. The graph with the tensors as vertices and the indices as edges is assumed to be a tree, i.e., to have no loops. The tree tensor network is then obtained by contracting over the auxiliary indices.
\ecl

As there does not appear to exist a firmly established mathematical notation for tree tensor networks, we give a formulation from scratch \bcl that turns out useful for the formulation, implementation and analysis of the numerical methods.\ecl

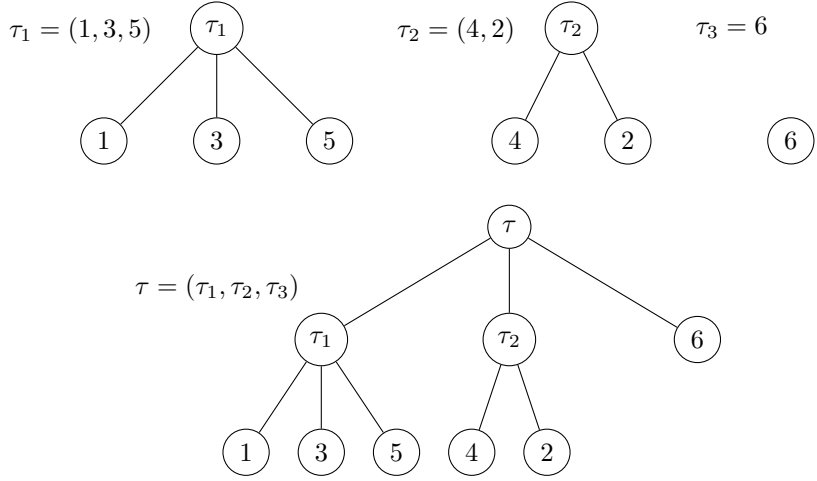
\begin{figure} [t]
	\label{fig:tree} 
	\begin{center}
	\begin{tikzpicture}
		\node[circle,draw]  { $\tau_1 $}
		child { node[circle,draw] {1} }
		child { node[circle,draw] {3} }
		child { node[circle,draw] {5} };
		
		% title
		\node[align=center,font=\bfseries, yshift=2em] (title) 
		at (current bounding box.west)
		{$ \tau_1 = (1,3,5) $};
	\end{tikzpicture}	
	\quad
	\begin{tikzpicture}
		\node[circle,draw] { $\tau_2$}
		child { node[circle,draw] {4} }
		child { node[circle,draw] {2} };
		
		% title
		\node[align=center,font=\bfseries, yshift=2em] (title) 
		at (current bounding box.west)
		{$ \tau_2 = (4,2) \qquad $ };
	\end{tikzpicture}
	\quad
	\begin{tikzpicture}
		\node[circle,draw] { $6$};
		 
		% title
		\node[align=center,font=\bfseries, yshift=4.3em] (title) 
		at (current bounding box.west)
		{$ \tau_3 = 6 \qquad $ };
	\end{tikzpicture}
	
	\vspace{0.5cm}
	
	\begin{tikzpicture}[every node/.style={},level 1/.style={sibling distance=25mm},level 2/.style={sibling distance=10mm}]
		\node[circle,draw] { $\tau$ }
		child { node[circle,draw] { $\tau_1$ }
			child { node[circle,draw] {1} }
			child { node[circle,draw] {3} }
			child { node[circle,draw] {5} } }
		child { node[circle,draw] { $\tau_2$ }
			child { node[circle,draw] {4} }
			child { node[circle,draw] {2} } }
		child{ node[circle,draw] {6}};
		
		% title
		\node[align=center,font=\bfseries, yshift=2em] (title) 
		at (current bounding box.west)
		{$ \tau = (\tau_1, \tau_2, \tau_3) $ };
	\end{tikzpicture}
	
	  \caption{Graphical representation of a tree and three subtrees with the set of leaves $\mathcal{L}=\{1,2,3,4,5,6\}$.}
	\end{center}	
	
\end{figure} 

\begin{definition}[Ordered trees with unequal leaves]
	\bcl
	Let $\mathcal{L}$ be a given finite set, the elements of which are referred to as leaves.
	We  define the set $\mathcal{T}$ of trees  $\tau$ with the corresponding set of leaves $L(\tau)\subseteq \mathcal{L}$ recursively as follows:
	\ecl
	\begin{enumerate}[(i)]
		\item
		\bcl {\em Leaves are trees:} \ecl
			$ \mathcal{L} \subset \mathcal{T}$,\ \text{ and }\  $L(\ell) := \{\ell\}$ for each $\ell \in \mathcal{L}$.
		\item \bcl {\em Ordered $m$-tuples of trees with different leaves are again trees:} \ecl
			If, for some $m\ge 2$,
			$$
			\tau_1, \dots, \tau_m \in \mathcal{T} 
			 \quad \text{ with }\quad
			L(\tau_i ) \cap L(\tau_j ) = \emptyset \quad \forall i \neq j,
			$$
			\bcl then their ordered $m$-tuple is in $\mathcal{T}$:
			\ecl
			$$ \tau := (\tau_1, \dots, \tau_m) \in \mathcal{T} 
					, \quad \text{ and }\quad 
					L(\tau) := \dot{\bigcup}_{i=1}^m L(\tau_i) \ . 
			$$
	\end{enumerate}		
	%An element $\tau \in \mathcal{T}$ is called a tree.
\end{definition}
The graphical interpretation is that (i) leaves are trees and (ii) every other tree $\tau \in\mathcal{T}$ is formed by connecting a root to several trees with different leaves.  (Note that $m=1$ is excluded: the 1-tuple $\tau=(\tau_1)$ is considered to be identical to $\tau_1$.)
$L(\tau)$ is the set of leaves of the tree $\tau$. 

%We call its cardinality (i.e., the number of its leaves) the dimension of the tree:
%$$
%d(\tau) = \# L(\tau).
%$$
%We remark that the ordering of the tree (ordered $m$-tuples instead of unordered ones) plays no role for the tree tensor network defined below, but it is important for the time integration method.

			The trees $\tau_1,\dots,\tau_m$ are called direct subtrees of the tree $\tau= (\tau_1, \dots, \tau_m)$, which together with direct subtrees of direct subtrees of $\tau$ etc. are called the {\it subtrees} of $\tau$. We let $T(\tau)$ be the set of subtrees of a tree $\tau\in\mathcal{T}$, including~$\tau$. More formally, we set 
$$
\text{
$T(\ell):= \{\ell\}$ for $\ell\in\mathcal{L}$, \ and $\ T(\tau):=\{ \tau \} \, \dot \cup\; \dot{\bigcup}_{i=1}^m T(\tau_i)$ for $\tau = (\tau_1, \dots, \tau_m)$.
}
$$
In the graphical interpretation, the subtrees are in a bijective correspondence with the vertices of the tree, by assigning  to each subtree its root; see Figure~2.1.

On the set of trees $\mathcal{T}$ we define a partial ordering by writing, for $\sigma,\tau\in\mathcal{T}$,
\begin{align*}
&\sigma \le \tau \quad \text{if and only if} \quad \sigma \in T(\tau), 
\\
&\sigma < \tau \quad \text{if and only if} \quad \sigma \le \tau \ \text{ and } \ \sigma \ne \tau.
\end{align*}

\bcl
We now fix a maximal tree $\bar\tau\in\mathcal{T}$ (with $L(\bar \tau)=\mathcal{L}$). On this tree 
\ecl
we work with the following \bcl quantities\ecl:
\begin{enumerate}
\item To each leaf $\ell\in\mathcal{L}$ we associate a dimension $n_\ell$, a rank $ r_\ell \leq n_\ell$ and a {\it basis matrix} $\U_\ell \in \mathbb{R}^{n_\ell  \times r_\ell}$ of full rank $r_\ell$.
\item To every subtree $\tau = ( \tau_1, \dots, \tau_m)\le \bar\tau$ we associate a rank $r_\tau$ and a {\it connection tensor} $ C_\tau \in \mathbb{R}^{r_\tau \times r_{\tau_1} \times \dots \times r_{\tau_m}}$ of full multilinear rank $(r_\tau, r_{\tau_1},\ldots, r_{\tau_m})$. We set $r_{\bar\tau}=1$.
\end{enumerate}
This can be interpreted as associating a tensor to the {\it root} of each subtree. In this way every vertex of the given tree carries either a matrix --- if it is a leaf --- or else a tensor whose order equals the number of edges leaving the vertex. \bcl In applications, the basis matrices correspond to the physical degrees of freedom and the connection tensors determine the correlation network.\ecl

With these data, a tree tensor network (TTN) is constructed in a recurrence relation that passes from the leaves to the root of the tree.

\pagebreak[3]
\begin{definition}[Tree tensor network]
For a given tree $\bar \tau \in \mathcal{T}$ \bcl and basis matrices $\U_\ell$ and connection tensors $C_\tau$ as described in 1. and 2. above, we recursively define a tensor $X_{\bar \tau}$ with a tree tensor network representation \ecl as follows: 
	\begin{enumerate}[(i)]
		\item
			\bcl For each leaf \ecl $\,\tau = \ell \in \mathcal{L}$, we set 
			$$X_\ell := \U_\ell^\top \in \R^{r_\ell \times n_\ell} \ .  $$
		\item
			\bcl If, for some $m\ge 2$, the tree $\tau = ( \tau_1, \dots, \tau_m)$ is a subtree of $\bar\tau $, \ecl then
			we set $n_\tau = \prod_{i=1}^m  n_{\tau_i}$ and $\I_\tau$ the identity matrix of dimension $r_\tau$, and
			\begin{align*}
				& X_\tau := C_\tau \times_0 \I_{\tau} \bigtimes_{i=1}^m \U_{\tau_i} 
				\in \mathbb{R}^{r_\tau \times n_{\tau_1} \times \dots \times n_{\tau_m}},
				\\
				& \U_{\tau} := \mat_0( X_\tau )^\top \in \mathbb{R}^{n_\tau \times r_\tau} \ .
			\end{align*}
	\end{enumerate}
\end{definition}
	
\bcl	For short, we refer to the tensor $X_{\bar\tau}$ as a {\em tree tensor network} on the tree $\bar\tau$. \ecl
	
The expression on the righthand side of the definition of $X_\tau$ in (ii) can be viewed as an $r_\tau$-tuple of $m$-tensors with the same basis matrices $\U_{\tau_i}$ but different core tensors $C_\tau(k,:)\in \mathbb{R}^{n_{\tau_1} \times \dots \times n_{\tau_m}}$ for $k=1,\dots,r_\tau$. The vectorizations of these $r_\tau$ $m$-tensors, which are of dimension $n_\tau$, form the columns of the matrix $\U_{\tau}$. The index $0$ in $\times_0$ and $\mat_0$ refers to the mode  of dimension $r_\tau$ of $X_\tau\in \mathbb{R}^{r_\tau \times n_{\tau_1} \times \dots \times n_{\tau_m}}$, which we count as mode~$0$. 
\bcl
The product $\times_0 \I_{\tau}$ is redundant in the definition of $X_\tau$, but we include it to emphasize the fact that $X_\tau$ is a tensor of order $m+1$. In the graphical representation, the edge $0$ is directed upward to the parent vertex and the edges $1,\ldots,m$ are directed downward to the subtrees; see Figure~2.1.

\bcl
This construction gives a data-sparse representation of a tensor with $\prod_{\ell\in\mathcal{L}} n_\ell$ entries.
For a rough bound of the memory requirements, let $d$ be the number of leaves and note that the number of vertices of the tree that are not leaves, is less than~$d$.
We set $n=\max_\ell n_\ell$  and $r=\max_\tau r_\tau$ and let $m+1$ be the maximal order of the connection tensors $C_\tau$. The basis matrices and connection tensors then have less than
$$
dnr + d r^{m+1} \ll n^d \quad \text{entries.}
$$

We note that the representation of $X_{\bar\tau}$ in terms of  the basis matrices $\U_\ell$ and connection tensors $C_\tau$ of full multilinear rank is not unique; cf.~\cite{UschV13}.
\ecl
It is favorable to work with orthonormal matrices,  so that each tensor $X_\tau$ for $\tau\le \bar\tau$ is in the Tucker format.

\begin{definition}[Orthonormal tree tensor network]
A tree tensor network $X_{\bar\tau}$ (more precisely, its representation in terms of the matrices $\U_\tau$) is called \emph{orthonormal}  if for each subtree $\tau<\bar\tau$, the matrix  $\U_\tau$ has orthonormal columns.
%, i.e.,
%$$ \U_\tau^\top \U_\tau = \I_\tau \quad \forall \tau \in \mathcal{T} \ .$$
\end{definition}

The following is a key lemma.

\begin{lemma} \label{lem:ttn-orth}
	For a tree $\tau = (\tau_1, \dots, \tau_m) \in \mathcal{T}$,	let the matrices $\ \U_{\tau_1}, \dots, \U_{\tau_m}$ have orthonormal columns. Then, the matrix  $\U_\tau$ has orthonormal columns if and only if the matricization $\ \mat_0(C_\tau)^\top \in \mathbb{R}^{ r_{\tau_1} \dots r_{\tau_m}  \times r_\tau }$ has orthonormal columns. 
\end{lemma}

\begin{proof} We have, by the definition of $\U_\tau$ and $X_\tau$ and the unfolding formula \eqref{unfolding},
$$
\U_\tau^\top = \mat_0 (X_\tau) = \I_\tau \mat_0(C_\tau) \bigotimes_{i=1}^m \U_{\tau_i} ^\top = \mat_0(C_\tau) \bigotimes_{i=1}^m \U_{\tau_i} ^\top.
$$
It follows that
$$
\U_\tau^\top \U_\tau = \mat_0(C_\tau) \Big( \bigotimes_{i=1}^m \U_{\tau_i} ^\top \U_{\tau_i} \Big) \mat_0(C_\tau)^\top=
\bigl( \mat_0(C_\tau)^\top\bigr)^\top \mat_0(C_\tau)^\top,
$$
which proves the result.
\end{proof} 

We observe that due to the recursive definition of a tree tensor network it thus suffices to require that for each leaf the matrix $\U_\ell$ and for every other subtree $\tau<\bar\tau$ the matrix $\mat_0(C_\tau)^\top$ have orthonormal columns.  
{\em Unless stated otherwise, we intend the tree tensor networks to be orthonormal in this paper.}

The orthonormality condition of Lemma~\ref{lem:ttn-orth} is very useful, because it reduces the orthonormality condition for the large, recursively constructed and computationally inaccessible matrix $\U_\tau\in  \mathbb{R}^{n_{\tau_1} \dots n_{\tau_m} \times r_\tau}$ to the orthonormality condition for the smaller, given matrix $ \mat_0(C_\tau)^\top \in \mathbb{R}^{ r_{\tau_1} \dots r_{\tau_m}  \times r_\tau }$. To our knowledge, the first use of this important property was made in the chemical physics literature in the context of the multilayer MCTDH method \cite{WaT03}. 

A further consequence of Lemma~\ref{lem:ttn-orth} is that {\em every (non-orthonormal) tree tensor network has an orthonormal representation}. This is shown using a QR decomposition of non-ortho\-normal matrices $\mat_0(C_{\tau_i})^\top=\Q_{\tau_i}\mathbf{R}_{\tau_i}$ and including the  factor $\mathbf{R}_{\tau_i}$ in the tensor $C_\tau$ of the parent tree $\tau=(\tau_1,\dots,\tau_m)$:
\bcl $C_{\tau_i}$ is changed to $\ten_0(\Q_{\tau_i}^\top)$ and $C_\tau$ is changed to $C_\tau \bigtimes_{i=1}^m \mathbf{R}_{\tau_i}$. This is done recursively from the leaves to the root.
We note that the property of $C_\tau$ to be of full multilinear rank $(r_\tau, r_{\tau_1},\ldots, r_{\tau_m})$
does not change under these transformations.\ecl

We rely on the following property, which can be proved as in \cite{UschV13}, where binary trees are considered (this corresponds to the case $m=2$ above); see also \cite{FaHN15}.

\begin{lemma} \label{lem:mf}
\bcl
Let a tree $\bar \tau\in\mathcal{T}$ be given together with dimensions $(n_\ell)_{\ell\in L(\bar\tau)}$ and ranks $(r_\tau)_{\tau\in T(\bar\tau)}$. The
set $\mathcal{M}_{\bar\tau}=\mathcal{M}(\bar\tau, (n_\ell)_{\ell\in L(\bar\tau)}, (r_\tau)_{\tau\in T(\bar\tau)})$ of 
 tensors with an orthonormal tree tensor network representation of the given dimensions and ranks
 is a smooth embedded manifold in the tensor space  $\R ^{\times_{\ell\in L(\bar\tau)} n_\ell}$.
 \ecl
\end{lemma}

%\section{Dynamical low-rank approximation}

\section{Extended Tucker Integrator}
\label{sec:eti}

In Subsection~\ref{subsec:nti} we recapitulate the algorithm of the Tucker tensor integrator of \cite{Lu15,LuVW18}, and in Subsection~\ref{subsec:eti} we extend the algorithm to $r$-tuples of Tucker tensors with the same basis matrices. This will be a  basic building block for the tree tensor network integrator derived in the next section.

\subsection{Tucker tensor integrator} \label{subsec:nti}
Let $\mathcal{M}_\textbf{r}$ be the manifold of tensors of order $d$ with fixed multi-linear rank \textbf{r}=$(r_1,\dots,r_d)$.
Let us consider an approximation $Y^0 \in  \mathcal{M}_\textbf{r}$ to the initial data~$A^0$,
\begin{equation*}
	 Y^0 = C^0 \bigtimes_{i=1}^d \U_i^0 \in \R^{ n_1 \times \dots \times n_d},
\end{equation*}
\bcl 
where the basis matrices $\U_i^0\in\R^{n_i\times r_i}$ have orthonormal columns and the core tensor $C^0\in \R^{r_1\times\dots\times r_d}$ has full multilinear rank \textbf{r}=$(r_1,\dots,r_d)$.
\ecl
The Tucker integrator is a numerical procedure that gives, after $(d+1)$ substeps, an approximation in Tucker tensor format, $Y^1\in\M_\textbf{r}$,  to the full solution $A(t_1)$ after a time step $t_1=t_0+h$. The procedure is repeated over further time steps to yield approximations $Y^n\in\M_\textbf{r}$ to $A(t_n)$.
At each substep of the algorithm, only one factor of the Tucker representation is updated while the others, with the exception of the core tensor, remain fixed. The essential structure of the algorithm 
\bcl
can be summarized as follows: Starting from $Y^0 = C^0 \bigtimes_{i=1}^d \U_i^0$ in factorized form, initialize $C_0^0 = C^0$.
%in the following algorithmic scheme.
\begin{align*}
	&\text{For $i=1, \ldots, d$, update $\U_i^0 \rightarrow \U_i^1$ and modify $C_{i-1}^0 \rightarrow C_i^0$.}
		\\
	&\text{Update } C_d^0 \rightarrow C^1.
\end{align*}
This yields the result after one time-step, $Y^1 = C^1 \bigtimes_{i=1}^d \U_i^1$ in factorized form.
\ecl

%\bigskip
%\begin{algorithm}[H]
%	\caption{Time step of the Tucker integrator}
%	
%	\SetAlgoLined
%	\KwData{ Core tensor $C^0$ and orthonormal-basis matrices $\U_i^0$  of
%		$Y^0 = C^0 \bigtimes_{i=1}^d \U_i^0$, righthand side function $F(t,Y)$, $t_0, t_1$ 	
%	}
%	\KwResult{Core tensor $C^1$ and orthonormal-basis matrices $\U_i^1$ of
%	$Y^1 = C^1 \bigtimes_{i=1}^d \U_i^1$ }
%	\Begin{
%		Set $C_0^0 = C^0$
%
%		\For{$i=1 \dots d$}{
%			Update $\U_i^0 \rightarrow \U_i^1$, 
%			Modify $C_{i-1}^0 \rightarrow C_i^0$
%		}
%		Update $C_d^0 \rightarrow C^1$
%	}
%\end{algorithm}
%
%\bigskip
%The update process in the Tucker integrator is the most intensive and technical part of the algorithm. In order to reduce the complexity 
To simplify
the description, we introduce subflows $\Phi^{(i)}$  and $\Psi$, corresponding to the update of the basis matrices and of the core tensor, respectively. We set $r_{\neg i}= \prod_{j\ne i} r_i$ \bcl and use the notation $\S_i^{0,\top}=\bigl(\S_i^{0}\bigr)^\top$, $\textbf{V}_i^{0, \top}= \bigl(\textbf{V}_i^{0}\bigr)^\top$ etc.\ecl

\bigskip
\begin{algorithm}[H]
\label{alg:Phi-i}
	\caption{Subflow $\Phi^{(i)}$ }
	
	\SetAlgoLined
	\KwData{ 
		$Y^0 = C^0 \bigtimes_{j=1}^d \U_j^0 \text{ in factorized form}, F(t,Y), t_0, t_1$ 	
	}
	\KwResult{$Y^1 = C^1 \bigtimes_{j=1}^d \U_j^1$ in factorized form}
	\Begin{
		set $ \U_j^1 = \U_j^0 \quad \forall j \neq i$
		
		compute the QR decomposition $ \mat_i( C^0)^\top = \Q_i^0 \S_i^{0,\top} \in \R^{r_{\neg i}\times r_i}$ 	
			
		%set $\textbf{V}_i^{0,\top} = \mat_i( \text{Ten}_i( \Q_i^{0,\top}) \bigtimes_{l=i+1}^d \U_l^{0,\top})\in \R^{r_i\times r_{<i}n_{> i}}$
		
		set $ \K_i^0 = \U_i^0 \S_i^0 \in \R^{n_i\times r_i}$ 
		
		%set $ \textbf{Y}_{[i]}^{+}(t) = \K_i(t) \textbf{V}_i^{0, \top} $
		
		solve the $n_i\times r_i$ matrix differential equation\\
		\qquad $ \dot{\K}_i (t) = \F_i(t, {\K}_i(t)) $ \ 
		with initial value $\K_i(t_0) = \K_i^0$ \\
		$\qquad$and return $ \K_i^1 = \K_i(t_1)$; here\\
		$\qquad\qquad \F_i(t,\K_i) = \mat_i ( F(t, \text{Ten}_i( \K_i(t) \textbf{V}_i^{0, \top})
		) \textbf{V}_i^0$ with\\
		$\qquad\qquad\textbf{V}_i^{0,\top} = \mat_i( \text{Ten}_i( 
		\Q_i^{0,\top}) \bigtimes_{j\ne i} \U_j^{0}
		)
		$
		
		compute the QR decomposition $ \K_i^1 = \U_i^1 \widehat{\S}_i^1$ 

                 solve the $r_i\times r_i$ matrix differential equation\\
		\qquad $ \dot{\S}_i (t) = - \widehat\F_i(t, {\S}_i(t)) $ \ 
		with initial value $\S_i(t_0) = \widehat{\S}_i^1$ \\
		$\qquad$and return $ \widetilde{\S}_i^0 = \S_i( t_1) $; here\\
		 $\qquad\qquad\widehat\F_i(t, {\S}_i) = \U_i^{1,\top}\F_i(t,\U_i^1{\S}_i)$

		%Set $ \textbf{Y}_{[i]}^{-}(t) = \U_i^1 \S_i(t) \textbf{V}_i^{0,\top} $
		
%		solve 
%		$ 
%		\dot{\S}_i (t) = 
%		-\U_i^{1,\top} 
%		\mat_i(
%		F(
%		t, 
%		\text{Ten}_i( \textbf{Y}_{[i]}^{-}) \bigtimes_{k=1}^{i-1} \U_k^1 )
%		\bigtimes_{k=1}^{i-1} \U_k^{1,\top}
%		) 
%		\textbf{V}_i^0, 
%		$
%		
%		with initial value $\S_i(t_0)=\hat{\S}_i^1$ and return $ \tilde{\S}_i^0 = \S_i( t_1) $
		
		set $ C^1 = \text{Ten}_i ( \widetilde{\S}_i^0 \Q_i^{0,\top}  )$ 
		
%		set $ Y^1 = C^1 \bigtimes_{j=1}^d \U_j^1 $
	}
\end{algorithm}

\bigskip
\bcl
The following remarks will be used in the next subsection:

$\bullet$ Instead of the QR decomposition, any orthogonal decomposition (e.g., SVD) can be used that yields $\Q_i^0$ and $\U_i^1$ with orthonormal columns. This yields the same tensor $Y^1$, albeit in a different factorization.

$\bullet$ In the case where $F(t,Y)$ does not depend  on either $t$ or  $Y$, we obtain the same result $\widetilde{\S}_i^0$ if we replace the differential equation  for $\S$  with the negative sign from $t_0\to t_1$, with initial value $\S_i(t_0) = \widehat{\S}_i^1$, by the differential equation with the positive sign for $\S$ backward in time from $t_1\to t_0$ with final value $\S_i(t_1) = \widehat{\S}_i^1$:

\medskip
  solve  $\dot{\S}_i (t) =  \widehat\F_i(t, {\S}_i(t)) $ \ 
		with final value $\S_i(t_1) = \widehat{\S}_i^1$
		and return $ \widetilde{\S}_i^0 = \S_i( t_0) $.
		
\medskip\noindent
For a general  non-autonomous function  $F(t,Y)$, the forward and  backward formulations are no longer equivalent, but the robust convergence result of \cite{LuVW18} holds equally true for both formulations.

\medskip
We further remark that the differential equations for $\K_i(t)$ and $\S_i(t)$ need to be solved numerically by a standard integrator, unless $F(t,Y)=\dot A(t)$ is independent of~$Y$.
\ecl

\medskip
The remaining subflow $\Psi$ describes the final update process of the core. 
\\

\begin{algorithm}[H]
	\caption{Subflow $\Psi$ }
	
	\SetAlgoLined
	\KwData{ 
		$Y^0 = C^0 \bigtimes_{j=1}^d \U_j^0 \text{ in factorized form}, F(t,Y), t_0, t_1$ 	
	}
	\KwResult{$Y^1 = C^1 \bigtimes_{j=1}^d \U_j^1$  in factorized form}
	\Begin{
		set $ \U_j^1 = \U_j^0 \quad \forall j=1, \dots , d$.
		
		%Set $ \L^{0,T} = \mat_d( C^0) $
		
		solve the $r_1\times\dots\times r_d$ tensor differential equation\\
		$\qquad \dot C(t) = \widetilde F(t, C(t))$ \ with initial value $C(t_0)=C^0$\\
		$\qquad$and return $C^1=C(t_1)$; here\\
		$\qquad\qquad \widetilde F(t, C) = F(t, C \bigtimes_{j=1}^d \U_j^1) \bigtimes_{j=1}^d \U_j^{1,\top}$

%		
%		
%		
%		$  
%		\dot{\L}^T(t) = 
%		\U_d^{1,T} 
%		\mat_d(
%		F( 	t, 
%		\text{Ten}_d( \U_d^1 \L^T(t)) 
%		\bigtimes_{k=1}^{d-1} \U_k^1
%		) 
%		\bigtimes_{k=1}^{d-1} \U_k^{1,T}
%		),
%		$
%		
%		with initial data $\L^T(t_0) = \L^{0,T}$ and return $ \L^{1,T} = \L^T(t_1) $
%		
%		set $ C^1 = \text{Ten}_d(\L^{1,T})$		
%		set $ Y^1 = C^1 \bigtimes_{j=1}^d \U_j^1 $
	}
\end{algorithm}

\bigskip
Finally, the result of the Tucker tensor integrator after one time step can be expressed in a compact way  as
\begin{equation} \label{NestedTuckerInt}
	Y^1 = \Psi \circ \Phi^{(d)} \circ \dots \circ \Phi^{(1)}(Y^0) \ .
\end{equation}

We refer the reader to \cite{Lu15,LuVW18} for a detailed derivation and major properties of this Tucker tensor integrator.

\medskip
The efficiency of the implementation of this algorithm depends on the possibility to evaluate the functions $\F_i$ without explicitly forming the large slim matrix $ \textbf{V}_i^0 \in \R^{n_{\neg i}\times r_i}$ and the tensor $\text{Ten}_i( \K_i(t) \textbf{V}_i^{0, \top})\in \R^{n_1\times\dots\times n_d}$. This is the case if 
$F(t, C \bigtimes_{j=1}^d \U_j)$ is a linear combination of Tucker tensors of moderate rank whose factors can be computed directly from the factors $C$ and $\U_j$ without actually computing the entries of the Tucker tensor.

\subsection{Extended Tucker Integrator} \label{subsec:eti}
We consider the case of a 
 Tucker tensor of order $1+d$ where the first basis matrix in the decomposition of the initial data is the identity matrix of dimension $r \times r$,
$$ Y^0 =  C^0 \times_0 \I_r \bigtimes_{i=1}^d \U_i^0 \in \R^{r \times n_1 \times \dots \times n_d},$$
as appears in the recursive construction of orthonormal tree tensor networks.
This can be viewed as a collection of $r$ Tucker tensors in $\R^{ n_1 \times \dots \times n_d}$ with the same basis matrices $\U_i^0$. Recalling (\ref{NestedTuckerInt}), the action of the Tucker integrator after one time step can be represented as 
\begin{equation*}
	Y^1 = \Psi \circ \Phi^{(d)} \circ \dots \circ \Phi^{(1)} \circ \Phi^{(0)}(Y^0) \ .
\end{equation*} 	 
The following result \bcl simplifies the computation.\ecl

\begin{lemma}
	\bcl With an appropriate choice of orthogonalization, the action of the subflow $\Phi^{(0)}$ on $Y^0$ becomes trivial
	\ecl
	, i.e.,
	$$ \Phi^{(0)}(Y^0) = Y^0 .$$
\end{lemma}

\begin{proof}
	In the first step of the subflow $\Phi^{(0)}$, we matricize the core tensor $C^0$ in the zero mode and we perform a QR decomposition,
	$$ \mat_0( C^0)^\top = \Q_0^0 \S_0^{0,\top}.$$
	We define
	$\textbf{V}_0^{0,\top} = \mat_0( \text{Ten}_0( \Q_0^{0,\top}) \bigtimes_{l=1}^d \U_l^{0,\top}) $
	and we set
	$ \K_0^0 = \I_r \S_0^0 \in \R^{r \times r} .$
	The next step consists of solving the differential equation 
	\begin{align*}
	&\dot{\K}_0 (t) = 
	\mat_0 ( 
	F(
	t, 
	\text{Ten}_0( \K_0(t) \textbf{V}_0^{0, \top} ) 
	%\bigtimes_{k=1}^{i-1} \U_k^1 )
	%\bigtimes_{k=1}^{i-1} \U_k^{1,\top}
	) 
	\textbf{V}_0^0,
	\\
	&\K_0(t_0) = \K_0^0.
	\end{align*}
	We define
	$ \K_0^1 = \K_0(t_1) \in \R^{r \times r}  $
	\bcl
	and we use the trivial orthogonal decomposition (it is irrelevant that this is not the QR decomposition), 
	\ecl
	$$ \K_0^1 = \I_r \K_0^1 .$$
	Finally, we solve the differential equation
	\begin{align*}
	&\dot{\S}_0 (t) = 
	\I_r^{\top} 
	\mat_0(
	F(
	t, 
	\text{Ten}_0( \I_r \S_0(t) \textbf{V}_0^{0,\top} ) 
	%\bigtimes_{k=1}^{i-1} \U_k^1 )
	%\bigtimes_{k=1}^{i-1} \U_k^{1,\top}
	) 
	\textbf{V}_0^0,
	\\
	& 
	\S_0(t_1) = \K_0^1.
	\end{align*}
	\bcl This is the same differential equation as for $\K_0$, now solved backwards in time, and so we have
	$$ \widetilde \S_0^0 =\S_0(t_0) = \K_0(t_0) = \S_0^0 .$$
	\ecl
	Therefore, $C^1=C^0$ and we conclude that $\Phi^{(0)}(Y^0) = Y^0$.
\end{proof}

We can now introduce the extended Tucker integrator \bcl as given by Algorithm~\ref{alg:extended} below.  The adjective `extended' refers to the fact that we now approximate the solution to a differential equation for $r$ Tucker tensors with the same basis matrices rather than just one Tucker tensor. This extension is nontrivial but turns out to be remarkably simple.\ecl
\\

\begin{algorithm}[H]
\label{alg:extended}
	\caption{Extended Tucker Integrator}
	
	\SetAlgoLined
	\KwData{ 
		$Y^0 = C^0 \times_0 \I_r \bigtimes_{j=1}^d \U_j^0 \text{in factorized form}, F(t,Y), t_0, t_1$ 	
	}
	\KwResult{$Y^1 = C^1 \times_0 \I_r \bigtimes_{j=1}^d \U_j^1$ in factorized form}
	\Begin{
		Set $ Y^{[0]} = Y^0$
		
		\For{$i=1 \dots d$}{
			%Update $\U_i^0 \rightarrow \U_i^1$
			Compute $ Y^{[i]} = \Phi^{(i)}( Y^{[i-1]}) $ in factorized form \bcl by Algorithm 1 \ecl
		}
		Compute $ Y^1 = \Psi(Y^{[d]})$ in factorized form \bcl by Algorithm 2 \ecl
	}
\end{algorithm}

%\pagebreak
\section{Recursive tree tensor network integrator}
\label{sec:ttni}
We now come to the central algorithmic section of this paper. We derive an integrator for orthonormal tree tensor networks,
\bcl
which in every time step updates the orthonormal-basis matrices $\U_\ell(t)$ of the leaves and the orthonormality-constrained connection tensors $C_\tau(t)$ of the other vertices of the tree. This is done  without ever computing the entries of the high-dimensional tensor $Y(t)$ that has the tree tensor network representation given by the factors $\U_\ell(t)$ and $C_\tau(t)$ and (approximately) solves the projected differential equation \eqref{eq:projEq} --- provided that the function $F$  can be evaluated at the tree tensor network $Y(t)$ using only its factors.
\ecl

\subsection{Derivation}
Let $\bar\tau\in\mathcal{T}$ be a given tree with the set of leaves $L(\bar\tau)=\{1,\dots,d\}$ and $(r_\tau)_{\tau\in T(\bar\tau)}$ a specified family of tree ranks, where we assume $r_{\bar\tau}=1$. 
%and identify $\mathbb{R}^{1\times n_1 \times \dots \times n_d}=\mathbb{R}^{n_1 \times \dots \times n_d}$.
For  each subtree $\tau = (\tau_1, \dots, \tau_m) \in T(\bar\tau)$ we introduce the space
\begin{equation} \label{V-tau}
\mathcal{V}_\tau := \mathbb{R}^{r_\tau \times n_{\tau_1} \times \dots \times n_{\tau_m}} \ .
\end{equation}

In the following, we associate to each subtree $\tau$ of the given tree $\bar\tau$ a tensor-valued function $F_{\tau}:[0,t^*]\times \mathcal{V}_\tau \to \mathcal{V}_\tau$. Its actual recursive construction, starting from the root with $F_{\bar\tau}=F$ and passing to the leaves, will be given in the next subsection.

%The starting point to build an approximation $Y^1$ in tree tensor format to the full solution $A(t_1)$  is the 
Consider a tree $\tau = (\tau_1, \dots, \tau_m)$ and
%By construction, there exists 
an extended Tucker tensor $Y_\tau^0$ associated to the tree $\tau$,
$$ Y_{\tau}^0  = C_\tau^0 \times_0 \I_\tau \bigtimes_{i=1}^m \U_{\tau_i}^0 \ . $$
Applying the extended Tucker integrator with the function $F_\tau$ we have that
$$ 	Y_\tau^1 = \Psi_\tau \circ \Phi^{(m)}_\tau \circ \dots \circ \Phi^{(1)}_\tau (Y_\tau^0) \ .$$
We recall that the subflow $\Phi^{(i)}_\tau$ gives the update process of the basis matrix $\U_{\tau_i}^0 \in \R^{n_{\tau_i} \times r_{\tau_i}}$. The extra subscript $\tau$ indicates that the subflow is computed for the function $F_\tau$.

We have two cases: 

\begin{enumerate}[(i)]
	\item
		If $\tau_i$ is a leaf, we directly apply the subflow $\Phi^{(i)}_\tau$ and update the basis matrix. 
	\item
		Else, we apply $\Phi^{(i)}_\tau$ only approximately (but call the procedure still $\Phi^{(i)}_\tau$). We tensorize the basis matrix and we construct new initial data $Y_{\tau_i}^0$ and a function $F_{\tau_i}$. We iterate the procedure in a recursive way, reducing the dimensionality of the problem at each recursion.
		\end{enumerate}

We are now in a position to formulate the recursive tree tensor network (TTN) integrator. It has the same general structure as the extended Tucker integrator as given by Algorithm~\ref{alg:extended}.

The difference to the extended Tucker integrator is that now the subflow $\Phi_\tau^{(i)}$ is no longer the same, but it recursively uses the TTN integrator for the subtrees. This approximate subflow is defined in close analogy to the subflow $\Phi^{(i)}$ for Tucker tensors, but the first differential equation is solved only approximately unless $\tau_i$ is a leaf. \bcl We remark that in the following Algorithm~\ref{alg:Phi-tau-i} the tensors $Y_{\tau_i}$ correspond to a tensorization of the matrices $\K_i$ in Algorithm~\ref{alg:Phi-i}; see the next two subsections for details of this correspondence.\ecl
\bigskip

\begin{algorithm}[H]
\label{alg:ttn}
	\caption{Recursive TTN Integrator}
	
	\SetAlgoLined
	\KwData{ 
		tree $\tau=(\tau_1,\dots,\tau_m)$, TTN in factorized form\\
		$\qquad Y_\tau^0 = C_\tau^0 \times_0 \I_\tau \bigtimes_{j=1}^m \U_{\tau_j}^0 
		\text{ with }  \U_{\tau_j}^0 =\mat_0(X_{\tau_j}^0)^\top$,\\
		$\qquad\text{function }F_\tau(t,Y_\tau), t_0, t_1$ 	
	}
	\KwResult{TTN $\,Y_\tau^1 = C_\tau^1 \times_0 \I_\tau \bigtimes_{j=1}^m \U_{\tau_j}^1
	         \text{ with } \U_{\tau_j}^1 =\mat_0(X_{\tau_j}^1)^\top$\\
		 $\qquad$in factorized form}
	\Begin{
		set $ Y_\tau^{[0]} = Y_\tau^0$
		
		\For{$i=1 \dots m$}{
			%Update $\U_i^0 \rightarrow \U_i^1$
			compute $ Y_\tau^{[i]} = \Phi_\tau^{(i)}( Y_\tau^{[i-1]}) $ in factorized form \bcl by Algorithm~\ref{alg:Phi-tau-i} \ecl
		}
		compute $ Y_\tau^1 = \Psi_\tau(Y_\tau^{[m]})$ in factorized form 
		\bcl by Algorithm~\ref{alg:Psi-tau} \ecl
	}
\end{algorithm}

\bigskip 

\begin{algorithm}[H] \label{alg:Phi-tau-i}
	\caption{Subflow $\Phi_\tau^{(i)}$ }
	
	\SetAlgoLined
	\KwData{ 
		tree $\tau=(\tau_1,\dots,\tau_m)$, TTN in factorized form\\
		$\qquad Y_\tau^0 = C_\tau^0 \times_0 \I_\tau \bigtimes_{j=1}^m \U_{\tau_j}^0 
		\text{ with }  \U_{\tau_j}^0 =\mat_0(X_{\tau_j}^0)^\top$,\\
		$\qquad\text{function }F_\tau(t,Y_\tau), t_0, t_1$ 	
	}
	\KwResult{TTN $\,Y_\tau^1 = C_\tau^1 \times_0 \I_r \bigtimes_{j=1}^m \U_{\tau_j}^1
	         \text{ with } \U_{\tau_j}^1 =\mat_0(X_{\tau_j}^1)^\top$\\
		 $\qquad$in factorized form}

	\Begin{
		set $ \U_{\tau_j}^1 = \U_{\tau_j}^0 \quad \forall j \neq i$
		
		compute the QR factorization $ \mat_i( C_\tau^0)^\top = \Q_{\tau_i}^0 \S_{\tau_i}^{0,\top}	$ \\[1mm]
					
		%set $ \K_i^0 = \U_i^0 \S_i^0 \in \R^{n_i\times r_i}$ 
		set $Y_{\tau_i}^0 = X_{\tau_i}^0 \times_0 \S_{\tau_i}^{0,\top}$   \\[2mm]
		
		%set $ \textbf{Y}_{[i]}^{+}(t) = \K_i(t) \textbf{V}_i^{0, \top} $
		
		{\bf if} $\tau_i=\ell$ is a leaf, {\bf then} solve the $n_\ell\times r_\ell$ matrix differential equation \\
		\qquad $ \dot Y_{\tau_i} (t) = F_{\tau_i}(t, Y_{\tau_i}(t)) $ \ 
		with initial value $Y_{\tau_i}(t_0) = Y_{\tau_i}^0$ \\
		$\qquad$and return $ Y_{\tau_i}^1 = Y_{\tau_i}(t_1)$\\
		{\bf else}\\
		$\qquad \text{compute }Y_{\tau_i}^1 = \text{\it Recursive TTN Integrator } (\tau_i, Y_{\tau_i}^0, F_{\tau_i}, t_0,t_1)$\\[2mm]
				
		%compute the QR factorization $ \K_i^1 = \U_i^1 \widehat{\S}_i^1$ 
		compute the QR decomposition $\mat_0(C_{\tau_i}^1)^\top = {\widehat {\mathbf{Q}}}_{\tau_i}^1  \widehat{\S}_{\tau_i}^1$, where \\
		$\qquad C_{\tau_i}^1$ is the connecting tensor of $Y_{\tau_i}^1$ \\[1mm]
		
		set $\U_{\tau_i}^1 = \mat_0(X_{\tau_i}^1)^\top$, where the TTN $X_{\tau_i}^1$ is obtained from $Y_{\tau_i}^1$ by \\
		$\qquad$replacing
		the connecting tensor with $\widehat C_{\tau_i}^1= \ten_0(\widehat {\mathbf{Q}}_{\tau_i}^{1,T})$  \\[1mm]

                 solve the $r_{\tau_i}\times r_{\tau_i}$ matrix differential equation\\
		\qquad $ \dot{\S}_{\tau_i} (t) = - \widehat\F_{\tau_i}(t, {\S}_{\tau_i}(t)) $ \ 
		with initial value $\S_{\tau_i}(t_0) = \widehat{\S}_{\tau_i}^1$ \\
		$\qquad$and return $ \widetilde{\S}_{\tau_i}^0 = \S_{\tau_i}( t_1) $; here\\
		 $\qquad \qquad \widehat\F_{\tau_i}(t, {\S}_{\tau_i}) = 
		 \U_{\tau_i}^{1,\top}\mat_0\bigl(F_{\tau_i}(t,X_{\tau_i}^1\times_0 {\S}_{\tau_i}^\top) \bigr)^\top$ \\[1mm]

		%Set $ \textbf{Y}_{[i]}^{-}(t) = \U_i^1 \S_i(t) \textbf{V}_i^{0,\top} $
		
%		solve 
%		$ 
%		\dot{\S}_i (t) = 
%		-\U_i^{1,\top} 
%		\mat_i(
%		F(
%		t, 
%		\text{Ten}_i( \textbf{Y}_{[i]}^{-}) \bigtimes_{k=1}^{i-1} \U_k^1 )
%		\bigtimes_{k=1}^{i-1} \U_k^{1,\top}
%		) 
%		\textbf{V}_i^0, 
%		$
%		
%		with initial value $\S_i(t_0)=\hat{\S}_i^1$ and return $ \tilde{\S}_i^0 = \S_i( t_1) $
		
		set $ C_\tau^1 = \text{Ten}_i ( \widetilde{\S}_{\tau_i}^0 \Q_{\tau_i}^{0,\top}  )$ 
		
%		set $ Y^1 = C^1 \bigtimes_{j=1}^d \U_j^1 $
	}
\end{algorithm}

\bigskip
The subflow $\Psi_\tau$ is the same as for the Tucker integrator, for the function $F_\tau$ instead of $F$.
\bigskip

\begin{algorithm}[H]
\label{alg:Psi-tau}
	\caption{Subflow $\Psi_\tau$ }
	
	\SetAlgoLined
	\KwData{ 
		tree $\tau=(\tau_1,\dots,\tau_m)$, TTN in factorized form\\
		$\qquad Y_\tau^0 = C_\tau^0 \times_0 \I_\tau \bigtimes_{j=1}^m \U_{\tau_j}^0 
		\text{ with }  \U_{\tau_j}^0 =\mat_0(X_{\tau_j}^0)^\top$,\\
		$\qquad\text{function }F_\tau(t,Y_\tau), t_0, t_1$ 	
	}
	\KwResult{TTN $\,Y_\tau^1 = C_\tau^1 \times_0 \I_\tau \bigtimes_{j=1}^m \U_{\tau_j}^1
	         \text{ with } \U_{\tau_j}^1 =\mat_0(X_{\tau_j}^1)^\top$\\
		 $\qquad$in factorized form}
	\Begin{
		set $ \U_{\tau_j}^1 = \U_{\tau_j}^0 \quad \forall j=1, \dots , m$.
		
		%Set $ \L^{0,T} = \mat_d( C^0) $
		
		solve the $r_\tau \times r_{\tau_1}\times\dots\times r_{\tau_m}$ tensor differential equation\\
		$\qquad \dot C_\tau(t) = \widetilde F_\tau(t, C_\tau(t))$ \ with initial value $C_\tau(t_0)=C_\tau^0$\\
		$\qquad$and return $C_\tau^1=C_\tau(t_1)$; here\\
		$\qquad \qquad \widetilde F_\tau(t, C_\tau) = F_\tau(t, C_\tau \bigtimes_{j=1}^m \U_{\tau_j}^1) \bigtimes_{j=1}^m \U_{\tau_j}^{1,\top}$
	}
\end{algorithm}

\bigskip
The differential equations for $Y_{\tau_i}(t)$, $\S_{\tau_i}(t)$ in Algorithm~\ref{alg:Phi-tau-i} 
and $C_\tau(t)$ in Algorithm~\ref{alg:Psi-tau} need to be solved approximately by a standard numerical integrator, unless $F(t,Y)$ is independent of $Y$ (which then implies that also the functions $F_\tau(t,Y)$ are independent of $Y$).

The efficiency of the implementation of this algorithm depends on the possibility to evaluate the functions $F_{\tau}$, ${\widehat\F}_{\tau}$ and $\widetilde F_\tau$  efficiently for all subtrees~$\tau$ of~$\bar\tau$, without explicitly forming large matrices or tensors whose dimension exceeds by far that of the basis matrices and connecting tensors. This is the case if $F$ maps TTNs into linear combinations of TTNs of moderate tree rank whose factors can be computed directly from the basis matrices and connecting tensors without actually computing the entries of the TTN.

%
%
%
%
%
%
%%%%%%%%%%%%%%%%%%%%%%%%%%%%%%%%%%%%%%%%%%%%%%
%\bigskip
%
%+++++++++++++++++++++++++++++++++++++++++++++++++++++++++++++++++++++++
%
%\begin{algorithm}[H]
%	\caption{Recursive TTN Integrator}
%	
%	\SetAlgoLined
%	\KwData{ 
%		$\tau =(\tau_1,\dots,\tau_m)\in \mathcal{T}, Y_{\tau}^0 \text{ in factorized form}, F_{\tau}, t_0, t_1$ 	
%	}
%	\KwResult{$Y_{\tau}^1$ in factorized form}
%	\Begin{
%		set $ Y_\tau^{[0]} = Y_\tau^0$
%		
%		\For{$i=1 \dots m$}{
%			\uIf{ $\tau_i$ is a leaf }{
%				compute $ Y_\tau^{[i]} = \Phi_\tau^{(i)}( Y_\tau^{[i-1]})$ in factorized form
%			}
%			\Else{
%			         compute $Y_{\tau_i}^{[i-1]}=\pi_{\tau_i}(Y_\tau^{[i-1]})$ in factorized form\\
%				 compute, in factorized form,\\
%				  $\qquad Y_{\tau_i}^{[i]} = \text{\it Recursive TTN Integrator}\,( \tau_i, Y_{\tau_i}^{[i-1]}, F_{\tau_i},t_0,t_1) $\\
%				 solve the $r_{\tau_i}\times r_{\tau_i}$ matrix differential equation\\
%				 $\qquad \dot\S_{\tau_i} = - \widehat \F_{\tau_i}(t, \S_{\tau_i})$\ with initial value 
%			}
%			
%		}
%		Set $ Y_\tau^1 = \Psi_\tau(Y_\tau^{[d]})$
%	}
%\end{algorithm}
%
%\bigskip
%We observe that the structure of the algorithm remains close to that of the extended Tucker integrator (where all subtrees are leaves). The fundamental change is in the recursion, which passes from the leaves to the root.
%
%++++++++++++++++++++++++++++++++++++++++++++++++++++++++++++++++++++++++++++++++++++++++++++++++++
%\pagebreak
\subsection{Constructing $F_\tau$ and $Y_\tau^0$ via restrictions/prolongations}
\label{subsec:F-tau}
%Consider the differential equation (\ref{eq:fullEq}) with the function 
%$$
%F: [0,t^*] \times \mathbb{R}^{n_1 \times \dots \times n_d}\to\mathbb{R}^{n_1 \times \dots \times n_d}
%$$ 
%and with  initial data in tree tensor format, 
%$$
%Y^0\in \mathcal{M}=\mathcal{M}(\bar\tau, (n_\ell)_{\ell=1}^d, (r_\tau)_{\tau\in T(\bar\tau)})\subset \mathbb{R}^{n_1 \times \dots \times n_d},
%$$ 
%where 
%Let $\bar\tau\in\mathcal{T}$ be a given tree with the set of leaves $L(\bar\tau)=\{1,\dots,d\}$ and $(r_\tau)_{\tau\in T(\bar\tau)}$ a specified family of tree ranks, where we assume $r_{\bar\tau}=1$. 
%%and identify $\mathbb{R}^{1\times n_1 \times \dots \times n_d}=\mathbb{R}^{n_1 \times \dots \times n_d}$.
%For  each subtree $\tau = (\tau_1, \dots, \tau_m) \in T(\bar\tau)$ we introduce the space
%$$ 
%\mathcal{V}_\tau := \mathbb{R}^{r_\tau \times n_{\tau_1} \times \dots \times n_{\tau_m}} \ .
%$$

In a recursion that passes from the root to the leaves of $\bar\tau$, we construct for each subtree $\tau$ of $\bar\tau$  the tensor-valued function $F_\tau$  that is used in the recursive TTN integrator.
%Let $\bar\tau$ be a fixed tree with leaves $L(\bar\tau) = \{1, \dots d\}$ and associated rank $ r_{\bar\tau} = 1$, which is the starting point of the recursive process. We assume $r_{\bar\tau}=1$, so 
%
%There exists a bijection,
%$$ \ten_{\bar\tau} : \mathcal{V}_{\bar\tau} \hookrightarrow \mathbb{R}^{n_1 \times \dots \times n_d} \ . $$
%We define the couple $(F_{\bar\tau}, Y_{\bar\tau}^0)$ as in the following
%\begin{align*}
%	& F_{\bar\tau}(t, Y_{\bar\tau}) = F(t, \ten_{\bar\tau}( Y_{\bar\tau} )),
%	\\
%	& \exists! \ Y_{\bar\tau}^0 \in \mathcal{V}_{\bar\tau}:  Y^0 = \ten_{\bar\tau}(Y_{\bar\tau}^0) \ .
%\end{align*}
We note that $\mathcal{V}_{\bar{\tau}}$ is isomorphic to $\mathbb{R}^{n_1 \times \dots \times n_d}$ (since $r_{\bar\tau}=1$) and
we start the construction by setting $F_{\bar\tau}=F:[0,t^*] \times \mathcal{V}_{\bar\tau} \rightarrow \mathcal{V}_{\bar\tau}$.
Given a subtree $\tau = (\tau_1, \dots, \tau_m) \in \mathcal{T}$, we now assume by induction that 
$$
 F_\tau: [0,t^*] \times \mathcal{V}_\tau \rightarrow \mathcal{V}_\tau	
$$
is already defined. For each $i=1,\dots,m$, we need to determine the tensor-valued function $F_{\tau_i}$ that appears in the subflow $\Phi_\tau^{(i)}$ of the recursive TTN integrator.
%
%We focus on the subtree $\tau_1$, since the same applies to the remaining direct subtrees.
%The first step in the recursive TTN integrator is an application of the extended Tucker integrator to the couple $( F_\tau, Y_\tau^0)$,
%$$ Y_\tau^1 = \Psi_\tau \circ \Phi^{(1)}_\tau\circ \dots \circ \Phi^{(1)}_\tau (Y_\tau^0) \ .$$
%If $\tau_1$ is a leaf, the procedure continues as in the extended Tucker integrator. Otherwise, we would like to take advantage of the new format and recursively update each factor of the decomposition. 
For the initial data
$ Y_\tau^0 = C_\tau^0 \times_0 \I_\tau \bigtimes_{i=1}^m \U_{\tau_i}^0$
the subflow $\Phi_\tau^{(i)}$ given by Algorithm~\ref{alg:Phi-tau-i} first computes the QR decomposition 
$$ 
\mat_i( C_\tau^0)^\top = \Q_{\tau_i}^0 \S_{\tau_i}^{0,\top} ,
$$
where $\Q_{\tau_i}^0 \in \R^{r_\tau r_{\neg \tau_i}\times r_{\tau_i}}$ with $r_{\neg \tau_i}=\prod_{j\ne i} r_{\tau_j}$
has orthonormal columns and 
$\S_{\tau_i}^{0} \in \R^{r_{\tau_i}\times r_{\tau_i}}$. Since
\begin{align*}
Y_\tau^0 = C_\tau^0 \times_0 \I_\tau \bigtimes_{i=1}^m \U_{\tau_i}^0 &= \text{Ten}_i( \S_{\tau_i}^0 \Q_{\tau_i}^{0,\top}) \times_0 \I_\tau \bigtimes_{i=1}^m \U_{\tau_i}^0
\\
&= \text{Ten}_i( \Q_{\tau_i}^{0,\top}) \times_0 \I_\tau \bigtimes_{j\ne i} \U_{\tau_j}^{0} \times_i ( \U_{\tau_i}^0 \S_{\tau_i}^0),
\end{align*}
we then have the SVD-like decomposition
\begin{equation} \label{Y-USV}
\mat_i( Y_\tau^0) = \U_{\tau_i}^0 \S_{\tau_i}^0 \V_{\tau_i}^{0,\top}
\end{equation}
with the (computationally inaccessible) matrix
$$\textbf{V}_{\tau_i}^{0} = \mat_i( \text{Ten}_i( \Q_{\tau_i}^{0,\top}) \times_0 \I_\tau\bigtimes_{j\ne i} \U_{\tau_j}^{0})^\top
\in \R^{r_\tau n_{\neg \tau_i} \times r_{\tau_i}}
$$
for $n_{\neg \tau_i}=\prod_{j\ne i} n_{\tau_j}=n_\tau/n_{\tau_i}$. We note that both $\U_{\tau_i}^0$ and $\V_{\tau_i}^0$ have orthonormal columns.

Like in Algorithm~\ref{alg:Phi-i} for the subflow $\Phi^{(i)}$ of the Tucker integrator, we consider the differential equation for ${\K}_{\tau_i} (t) \in \R^{n_{\tau_i}\times r_{\tau_i}}$,
\begin{equation} \label{eq:Ktau1}
	\begin{split}  
		& \dot{\K}_{\tau_i} (t) = 	\F_{\tau_i}(t,  \K_{\tau_i}(t)),
		\\
		& \K_{\tau_i}(t_0) = \U_{\tau_i}^0 \S_{\tau_i}^0 \ ,
	\end{split}	
\end{equation}
where 
$$ \F_{\tau_i}( \K_{\tau_i}) = \mat_i \bigl( F_\tau (t, 
		\text{Ten}_i( \K_{\tau_i} \V_{\tau_i}^{0, \top}))
		  \bigr)
		\textbf{V}_{\tau_i}^0.
$$
%with the (computationally inaccessible) matrix
%$$\textbf{V}_{\tau_i}^{0} = \mat_i( \text{Ten}_i( \Q_{\tau_i}^{0,\top}) \bigtimes_{j\ne i} \U_{\tau_j}^{0})^\top
%\in \R^{r_\tau n_{\neg \tau_i} \times r_{\tau_i}}
%$$
%for $n_{\neg \tau_i}=\prod_{j\ne i} n_{\tau_j}=n_\tau/n_{\tau_i}$.

Algorithm~\ref{alg:Phi-tau-i} uses this differential equation in tensorized form
%, letting $Y_{\tau_i}=\ten_i(\K_{\tau_i})$, \ecl 
and solves it approximately by recurrence down to the leaves. \bcl This relation is seen as follows: \ecl
By definition of the tree tensor network, there exists ${X}_{\tau_i}^0 \in \mathcal{V}_{\tau_i}$ such that
$$ 
\U_{\tau_i}^0 = \mat_0( {X}_{\tau_i}^0)^\top , \qquad \text{i.e.}, \quad  {X}_{\tau_i}^0 = \ten_0 ( \U_{\tau_i}^{0,\top} ).
$$
This implies that the initial condition in (\ref{eq:Ktau1}) can be rewritten as
$$ \K_{\tau_i}(t_0) = \U_{\tau_i}^0 \S_{\tau_i}^0 = \mat_0( {X}_{\tau_i}^0)^\top \S_{\tau_i}^0 = \mat_0({X}_{\tau_i}^0 \times_0 \S_{\tau_i}^{0,\top} )^\top ,
$$
which is the initial value chosen in Algorithm~\ref{alg:Phi-tau-i}.
We introduce
\begin{equation} Y_{\tau_i}(t) = \ten_0(\K_{\tau_i}(t)^\top), 
\qquad \text{i.e.}, \quad \K_{\tau_i}(t) = \mat_0( {Y}_{\tau_i}(t))^\top.
\end{equation}
By substitution, (\ref{eq:Ktau1}) can be rewritten as the differential equation \bcl that appears in Algorithm~\ref{alg:Phi-tau-i},\ecl
\begin{align*}
	&\dot{Y}_{\tau_i} (t)  = F_{\tau_i}(t, Y_{\tau_i}(t))
	%= \ten_0 ( \dot{\K}_{\tau_i}(t)^\top )
%	= 	\ten_0\Bigl(
%			\bigl(
%			\mat_1 ( \ 
%				F_\tau (t, 
%				\text{Ten}_1( \mat_0( Y_{\tau_i}(t))^\top  \V_{\tau_i}^{0, \top})
%				)
%			)
%			\textbf{V}_{\tau_i}^{0}  
%			\bigr)^\top
%		\Bigr)
%	,  
	\\
	& Y_{\tau_i}(t_0) = Y_{\tau_i}^0 := {X}_{\tau_i}^0 \times_0 \S_{\tau_i}^{0, \top}  \ ,
\end{align*}
where now
\begin{align}
\label{F-tau-i}
	 F_{\tau_i}(t, Y_{\tau_i}) &= \ten_0\bigl( \F_{\tau_i}(t, \mat_0(Y_{\tau_i})^\top\bigr) 
	 \\
	 \nonumber
	&=	
 	\ten_0\Bigl(
			\bigl(
			\mat_i ( \ 
				F_\tau (t, 
				\text{Ten}_i( \mat_0( Y_{\tau_i})^\top  \V_{\tau_i}^{0, \top})
				)
			)
			\textbf{V}_{\tau_i}^{0}  
			\bigr)^\top
		\Bigr) \ .
\end{align}
The construction of the tensor-valued function $F_{\tau_i}$ becomes more transparent by introducing the {\it prolongation}
\begin{equation}\label{prol}
\pi_{\tau,i}(Y_{\tau_i}) := \ten_i \bigl( (\V_{\tau_i}^{0}\mat_0(Y_{\tau_i}))^\top \bigr) \in \mathcal{V}_\tau  \quad\ \text{for } \  Y_{\tau_i}\in \mathcal{V}_{\tau_i} 
%\quad\text{for }\   \mat_0(Y_{\tau_i})^\top \in \R^{n_{\tau_i}\times r_{\tau_i}}.
\end{equation}
and the {\it restriction}
%Since  $\V_{\tau_i}^{0,\top} \V_{\tau_i}^{0}=\I$, the linear map $\pi_{\tau,i}$ admits the left inverse
\begin{equation}\label{rest}
\pi_{\tau,i}^{\dagger}(Z_\tau) := \ten_0 \bigl( (\mat_i(Z_\tau)\V_{\tau_i}^{0})^\top \bigr) \in \mathcal{V}_{\tau_i},
\quad\ \text{for }\   Z_\tau \in \mathcal{V}_\tau,
\end{equation}
where the tensorization $\ten_0$ is for a matrix in $\R^{r_{\tau_i}\times n_{\tau_i}}$ according to the dimensions of the subtrees of $\tau_i$. 

We note the following properties.

\begin{lemma}
\label{lem:prol-rest} Let $\tau=(\tau_1,\dots,\tau_m)$ and $i=1,\dots,m$.
The restriction $\pi_{\tau,i}^{\dagger}:\mathcal{V}_{\tau} \to \mathcal{V}_{\tau_i}$ is both a left inverse and the adjoint (with respect to the tensor Euclidean inner product) of the prolongation
$\pi_{\tau,i}:\mathcal{V}_{\tau_i} \to \mathcal{V}_{\tau}$, that is,
\begin{align}
\label{left-inv}
\pi_{\tau,i}^{\dagger} ( \pi_{\tau,i}(Y_{\tau_i})  ) &= Y_{\tau_i}  \qquad\ \text{for all } \  Y_{\tau_i}\in \mathcal{V}_{\tau_i} 
\\
\label{phi-adjoint}
\langle \pi_{\tau,i}(Y_{\tau_i}) , Z_\tau \rangle_{\mathcal{V}_{\tau} } 
&= \langle Y_{\tau_i}, \pi_{\tau,i}^{\dagger}(Z_\tau) \rangle_{\mathcal{V}_{\tau_i} } 
\quad\ \text{for all } \  Y_{\tau_i}\in \mathcal{V}_{\tau_i}, \, Z_\tau \in \mathcal{V}_{\tau}.
\end{align}
Moreover, $ \| \pi_{\tau,i}(Y_{\tau_i}) \|_{\mathcal{V}_{\tau} } = \| Y_{\tau_i} \|_{\mathcal{V}_{\tau_i} }$ and
$\| \pi_{\tau,i}^{\dagger}(Z_\tau) \|_{\mathcal{V}_{\tau_i} } \le \| Z_\tau  \|_{\mathcal{V}_{\tau} } $, where the norms are the tensor Euclidean norms.
\end{lemma}

\begin{proof} Since  $\V_{\tau_i}^{0,\top} \V_{\tau_i}^{0}=\I$, we obtain \eqref{left-inv}. Using that the tensorization $\ten_i$ is the adjoint of
the matricization $\mat_i$ for the Frobenius inner product and that taking transposes in both matrices of a Frobenius inner product does not change the inner product, we arrive at \eqref{phi-adjoint}. The norm equality  follows from the definition \eqref{prol} and the fact that the matrix $\V_{\tau_i}^0$ has orthonormal columns. The norm bound follows from
 \eqref{rest} on noting the general matrix norm inequality $\| \A \B \|_F \le \| \A \|_2 \,\| \B \|_F$  and the fact that  $\| \V_{\tau_i}^{0,\top} \|_2 =1$.
\end{proof}

We emphasize that the mappings $ \pi_{\tau,i}$ and $ \pi_{\tau,i}^\dagger$ depend on the initial data~$Y_\tau^0$. We observe that we can write \eqref{F-tau-i} more compactly as
$
F_{\tau_i} = \pi_{\tau,i}^\dagger \circ F_{\tau} \circ \pi_{\tau,i}.
$
For the initial data we find from \eqref{Y-USV} and \eqref{eq:Ktau1} that
$$
Y_{\tau_i}^0 =  \ten_0(\K_{\tau_i}(t_0)^\top) = \ten_0 \bigl( (\U_{\tau_i}^0 \S_{\tau_i}^0)^\top\bigr) = 
\ten_0 \bigl( (\mat_i(Y_\tau^0) \V_{\tau_i}^0)^\top ) = \pi_{\tau,i}^\dagger(Y_\tau^0).
$$
We thus arrive at the following.

\begin{definition} \label{def:F-tau}
For the given tensor-valued function $F_{\bar\tau}=F: [t_0,t^*] \times \mathcal{V}_{\bar\tau} \to  \mathcal{V}_{\bar\tau}$ and a tree tensor network $Y_{\bar\tau}^0 \in \mathcal{M}_{\bar\tau}$,
we define recursively for each tree $\tau=(\tau_1,\dots,\tau_m) \le \bar\tau$ and for $i=1,\dots,m$
\begin{align*}
F_{\tau_i} &= \pi_{\tau,i}^\dagger \circ F_{\tau} \circ \pi_{\tau,i}
\\[1mm]
Y_{\tau_i}^0 &= \pi_{\tau,i}^\dagger(Y_\tau^0).
\end{align*}
\end{definition}

\noindent
These are the nonlinear operators and initial data that are used in the recursive TTN integrator. We remark that their construction has a formal similarity to that of operators and functions in multilevel methods; cf.~\cite{Hac85}.

An important observation is the following.

\begin{lemma} \label{lem:rest-rank} 
If the initial tree tensor network $Y_{\bar\tau}^0$  has full tree rank $(r_\sigma)_{\sigma\le\bar\tau}$, then $Y_{\tau}^0$ has full tree rank $(r_\sigma)_{\sigma\le\tau}$
for every subtree $\tau\le\bar\tau$.
%Let $\tau=(\tau_1,\dots,\tau_m)\in T(\bar\tau)$ and $i=1,\dots,m$. 
%If $Y_\tau^0$  has full tree rank $(r_\sigma)_{\sigma\le\tau}$, then $Y_{\tau_i}^0=\pi_{\tau,i}^\dagger(Y_\tau^0)$ has full tree rank $(r_\sigma)_{\sigma\le\tau_i}$.
\end{lemma}

\begin{proof} Let $\tau=(\tau_1,\dots,\tau_m)\in T(\bar\tau)$ and $i=1,\dots,m$. 
From the above derivation we have, with the tensor network $X_{\tau_i}= \ten_0(\U_{\tau_i}^{0,\top})$ of full tree rank,
$$
Y_{\tau_i}^0= \ten_0((\U_{\tau_i}^0 \S_{\tau_i}^0)^\top) = X_{\tau_i}^0 \times_0  \S_{\tau_i}^{0,\top}.
$$
If $Y_\tau^0$  has full tree rank, then $\S_{\tau_i}^0$ is invertible, and hence also $Y_{\tau_i}^0 $ has full tree rank.
By induction we find that for every subtree $\tau\le\bar\tau$, the restricted initial tensor $Y_\tau^0$ has full tree rank.
\end{proof}

In terms of the manifold (see Lemma~\ref{lem:mf})
\begin{equation}\label{M-tau}
\mathcal{M}_{\tau}=\mathcal{M}(\tau, (n_\ell)_{\ell\in L(\tau)}, (r_\sigma)_{\sigma\le\tau})
\end{equation}
of  tree tensor networks for the tree $\tau\in\mathcal{T}$ of given dimensions $(n_\ell)_{\ell\in L(\tau)}$ and full tree rank $(r_\sigma)_{\sigma\le\tau}$,  Lemma~\ref{lem:rest-rank}  can be restated as saying that for $\tau=(\tau_1,\dots,\tau_m)$ and 
$Y_\tau^0\in\mathcal{M}_{\tau}$,  the restriction $\pi_{\tau,i}^\dagger (Y_\tau^0)$ is in $\mathcal{M}_{\tau_i}$. This statement is not true for an arbitrary $Y_\tau \in \mathcal{M}_{\tau}$ that is different from $Y_\tau^0$ (recall that the chosen restriction operator $\pi_{\tau,i}^\dagger$ depends on $Y_\tau^0$). In particular, a loss of rank occurs if for some $j$, the basis matrices are such that $\U_{\tau_j}^{0,\top}\U_{\tau_j}$ is a singular $r_{\tau_j}\times r_{\tau_j}$ matrix. However, for the prolongation we have the following.

\begin{lemma}\label{lem:prol-rank}
Let $\tau=(\tau_1,\dots,\tau_m)\in \mathcal{T}$ and $i=1,\dots,m$. If $\,Y_{\tau_i}\in \mathcal{M}_{\tau_i}$, then the prolongation 
$\pi_{\tau,i}(Y_{\tau_i})$ is in $\mathcal{M}_{\tau}$.
\end{lemma}

\begin{proof}
We have, using the definition of $\V_{\tau_i}^{0}$ and writing $\U_{\tau_i}:=\mat_0(Y_{\tau_i})^\top$,
\begin{align*}
\pi_{\tau,i}(Y_{\tau_i}) &= \ten_i \bigl( (\V_{\tau_i}^{0}\mat_0(Y_{\tau_i}))^\top \bigr) 
\\
&=  \ten_i \bigl( (\mat_i( \text{Ten}_i( \Q_{\tau_i}^{0,\top}) \times_0 \I_\tau\bigtimes_{j\ne i} \U_{\tau_j}^{0})^\top \mat_0(Y_{\tau_i}))^\top \bigr)
\\
&= \ten_i \bigl( \U_{\tau_i} \mat_i( \text{Ten}_i( \Q_{\tau_i}^{0,\top}) \times_0 \I_\tau\bigtimes_{j\ne i} \U_{\tau_j}^{0}) \bigr)
\\
&= \text{Ten}_i( \Q_{\tau_i}^{0,\top}) \times_0 \I_\tau\bigtimes_{j\ne i} \U_{\tau_j}^{0} \times_i \U_{\tau_i},
\end{align*}
which is of full tree rank.
\end{proof}

%We now observe that $\K_{\tau_i}(t_1) = \mat_0( E_{\tau_i}(t_1))^\top$ and the algorithm concludes  as in the subflow $\Phi^{(i)}$.

\subsection{QR decomposition} \bcl We now explain how the first QR decomposition in Algorithm~\ref{alg:Phi-i} is related to that of Algorithm~\ref{alg:Phi-tau-i}. 
We consider a tree $\tau = (\tau_1, \dots, \tau_m) \in \mathcal{T}$ and let $\tau$  take the role of $\tau_i$ in Algorithm~\ref{alg:Phi-tau-i} for ease of notation. In the extension of  Algorithm~\ref{alg:Phi-i} via \eqref{eq:Ktau1}, we would need the QR-decomposition of $\K_{\tau}^1\in \R^{n_\tau\times r_\tau}$, where we recall that $n_\tau = \prod_{j=1}^m n_{\tau_j}$ can get prohibitively large.
This difficulty is overcome  using that the tree tensor network is orthonormal: %by Lemma~\ref{lem:orth}, 
the QR-decomposition of the full matrix $\K_{\tau}^1$ is equivalent to the QR-decomposition of the matricization of a small core tensor. In fact, by construction we have that
$$ %\exists Y_{\tau} \in \mathcal{V}_{\tau} : 
\K_{\tau}^1 = \mat_0( Y_{\tau}^1)^\top ,  $$
where
$$ Y_\tau^1 = C_\tau^1 \times_0 \I_\tau \bigtimes_{i=1}^m \U_{\tau_i}^1 \ .$$
This implies that
$$ \K_{\tau}^1 = \Big( \bigotimes_{i=1}^m \U_{\tau_i}^1 \Big) \mat_0(C_\tau^1)^\top .$$
Since the Kronecker product of orthonormal matrices is orthonormal, it suffices to do a QR-decomposition of the comparatively small matrix $\mat_0(C_\tau^1)^\top \in \R^{ r_{\tau_1}\dots r_{\tau_m}\times r_\tau}$, as is done in Algorithm~\ref{alg:Phi-tau-i}.
\ecl

\subsection{Efficient computation of the prolongation $\pi_{\tau,i}$}
The construction  of the extremely large matrix $\V_{\tau_i}^{0,\top}$ appearing in the integrator must be avoided. In fact, recalling that
$$\textbf{V}_{\tau_i}^{0, \top} = \mat_i( \text{Ten}_i( \Q_{\tau_i}^{0,\top}) \bigtimes_{j\ne i} \U_{\tau_j}^{0})
\in \R^{r_{\tau_i} \times r_\tau n_{\neg \tau_i}},
$$
the mapping $\pi_{\tau,i}$ can be easily computed,
\begin{align*}
	\pi_{\tau,i}(Y) 
	& = \ten_i( \mat_0(Y)^\top \V_{\tau_i}^{0, \top} ) 
	\\
	& = \ten_i	(
				\mat_0(Y)^\top 
				\mat_i( \ten_i(\Q_{\tau_i}^{0,\top}) 
				\bigtimes_{j\ne i} \U_{\tau_j}^0 ) 
				)
	\\
	& = \ten_i	\big(
				\mat_i( \ten_i(\Q_{\tau_i}^{0,\top})
				\times_i \mat_0(Y)^\top 
				\bigtimes_{j\ne i} \U_{\tau_j}^0 ) 
				\big)
	\\
	& =  	\ten_i(\Q_{\tau_i}^{0,\top})
			\times_i \mat_0(Y)^\top 
			\bigtimes_{j \ne i} \U_{\tau_j}^0 \ .
\end{align*}
The action of the prolongation $\pi_{\tau,i}$ on the tree tensor $Y \in \mathcal{V}_{\tau_i}$ yields a new larger tree tensor in $\mathcal{V}_{\tau}$ with the core tensor $\ten_i(\Q_{\tau_i}^{0,\top})$.% and leaves as indicated before.   

\subsection{Efficient computation of the restriction $\pi_{\tau,i}^\dagger$}
\label{subsec:phi-tau-dagger}
The computation of the mapping $\pi^\dagger_{\tau_i}(Z_\tau)$ can be done efficiently by contraction if $Z_\tau$ is itself a tree tensor network on the tree $\tau=(\tau_1,\ldots,\tau_m)$ with the same dimensions $n_\tau$ and $n_{\tau_i}$ but possibly different  ranks $s_\tau$ and $s_{\tau_i}$ instead of $r_\tau$ and $r_{\tau_i}$, respectively:
%It suffices to use the definition of $\V_{\tau_i}^{0,\top}$ and multiply only small matrices or matricization of intermediate core tensors with low-rank. 
$$ 
Z_\tau = G_{\tau} \times_0 \I_{\tau} \bigtimes_{i=1}^m \W_{\tau_i} \in \mathcal{V}_{\tau}
$$
with the core tensor $G_{\tau}\in \R^{s_\tau\times s_{\tau_1}\times\dots\times s_{\tau_m}}$ (not necessarily of full multilinear rank)
and matrices $\W_{\tau_i}\in \R^{n_{\tau_i} \times s_{\tau_i}}$ (not necessarily of full rank).
%For sake of notation, we have changed the letter of the basis matrices in the definition. 
We have that
\begin{align*}
	\pi^\dagger_{\tau_i}(Z_\tau) 
	& = \ten_0(\V_{\tau_i}^{0, \top} \mat_i(Z_\tau)^\top)
	\\
	& = \ten_0(
			\Q_{\tau_i}^{0,\top}
			\big( \I_\tau \bigotimes_{j \ne i} \U_{\tau_j}^{0,\top} \big)  
			\mat_i(Z_\tau)^\top
		)
	\\
	& = \ten_0(
			\Q_{\tau_i}^{0,\top}
			\mat_i (G_\tau \times_0 \I_\tau \bigtimes_{j\ne i} ( \U_{\tau_j}^{0,\top} \W_{\tau_j} ) \times_i \W_{\tau_i})^\top )
%			\big( \I_\tau \bigotimes_{j \ne i} \U_{\tau_j}^{0,\top} \W_{\tau_j} \big) 
%			\mat_i(G_{\tau})^\top
%			\W_{\tau_i}^\top
 	\\
	& = \ten_0(
			\Q_{\tau_i}^{0,\top}
			\mat_i (G_\tau \times_0 \I_\tau \bigtimes_{j\ne i} ( \U_{\tau_j}^{0,\top} \W_{\tau_j} )) ^\top \W_{\tau_i}^\top )
\end{align*}
We define the small matrix
$$ \mathbf{R}_{\tau_i} := \Q_{\tau_i}^{0,\top}
   \mat_i (G_\tau \times_0 \I_\tau \bigtimes_{j\ne i} ( \U_{\tau_j}^{0,\top} \W_{\tau_j} ))^\top 
	%	 \big( \I_\tau \bigotimes_{j \ne i} \U_{\tau_j}^{0,\top} \W_{\tau_j} \big) 
	%	 \mat_i(G_{\tau})^\top
		\in \mathbb{R}^{r_{\tau_i} \times s_{\tau_i}},		
$$
\bcl
where we note that the product of the large matrices $\U_{\tau_j}^{0}$ and $\W_{\tau_j}$ (with $n_{\tau_j}$ rows, which is prohibitive unless $\tau_j$ is a leaf) can be computed recursively from small matrices: For a tree $\tau=(\tau_1,\dots,\tau_m)$, let
\begin{align*}
 \U_{\tau} &= \mat_0(C_\tau \times_0 \I_\tau \bigtimes_{i=1}^m \U_{\tau_i})^\top \in \R^{n_\tau\times r_\tau}, \\
 \W_{\tau} &= \mat_0(G_\tau \times_0 \I_\tau \bigtimes_{i=1}^m \W_{\tau_i})^\top \in \R^{n_\tau\times s_\tau}.
 \end{align*}
 Then, \eqref{unfolding} shows that the small matrix $ \U_{\tau}^\top  \W_{\tau} \in \R^{r_\tau\times s_\tau}$ equals
 \begin{equation}
 \label{UW}
  \U_{\tau}^\top  \W_{\tau} = \mat_0\bigl(C_\tau \bigtimes_{i=1}^m ( \U_{\tau_i}^\top  \W_{\tau_i} )^\top \bigr)  \,\mat_0(G_\tau)^\top,
 \end{equation}
 and hence this matrix can be computed recursively, passing from the leaves to the root~$\tau$. In this way we compute only  products $\U_\ell^\top \W_\ell$ for the leaves $\ell\in L(\tau)$ and products of matrices whose dimensions depend only on the ranks and orders of the connection tensors.

We return to the above expression for $\pi^\dagger_{\tau_i}(Z_\tau)$ and recall that by definition of the tree tensor network, there exists a tree tensor network $Z_{\tau_i}$ on the tree $\tau_i$ such that
$$
\W_{\tau_i} = \mat_0(Z_{\tau_i})^\top.
$$
We then have \ecl
\begin{align*}
	\pi^\dagger_{\tau_i}(Z_\tau) 
		& = \ten_0(\mathbf{R}_{\tau_i} \W_{\tau_i}^\top)
%		\\ 
%		& 
		= \ten_0(\mathbf{R}_{\tau_i} \mat_0(Z_{\tau_i}))
%		\\
%		& 
		= \ten_0(\mat_0(Z_{\tau_i} \times_0 \mathbf{R}_{\tau_i} ))
		\\
		& 
		= Z_{\tau_i} \times_0 \mathbf{R}_{\tau_i} \ .  
\end{align*}
\bcl
This implies that $\pi^\dagger_{\tau_i}(Z_\tau)$ differs from $Z_{\tau_i}$ only in that the (small) core tensor $G_{\tau_i}$ of $Z_{\tau_i}$ is replaced by 
$\widehat G_{\tau_i}= G_{\tau_i} \times_0  \mathbf{R}_{\tau_i} $, or equivalently,
$\mat_0(\widehat G_{\tau_i})= \mathbf{R}_{\tau_i} \mat_0(G_{\tau_i})$.
\ecl
%The action of the restriction $\pi^\dagger_{\tau_i}$ on the tree tensor network $Z_\tau \in \mathcal{V}_\tau$ is thus reduced to 
%the multiplication of the matricization of the core tensor $G_{\tau_i}$ of $Z_{\tau_i}$ with the tiny matrix~$\mathbf{R}_{\tau_i}$.

\bcl
\subsection{Computational complexity}
In the implementation of Algorithm~\ref{alg:Phi-tau-i} and Algorithm~\ref{alg:Psi-tau}, the matrices $\U_{\tau_i}=\mat_0(X_{\tau_i})^T$  (with superscripts $0$ and $1$), the tensors $Y_{\tau_i}$  and the values of the function $F_{\tau_i}$ are not stored and computed entrywise but only through their TTN factorization into basis matrices for the leaves and connection tensors. Furthermore, products of prohibitively large matrices such as $ \U_{\tau_i}^{1,\top}\mat_0\bigl(F_{\tau_i}(t,X_{\tau_i}^1\times_0 {\S}_{\tau_i}^\top) \bigr)^\top$ are reduced to matrix products of small matrices using \eqref{UW} recursively.
A count of the required operations and the required memory yields the following result. 

\begin{lemma}\label{lem:complexity}
Let $d$ be the order of the tensor $A(t)$ (i.e., the number of leaves of the tree $\bar\tau$), $l<d$ the number of levels (i.e., the height of the tree $\bar\tau$), and let $n=\max_\ell n_\ell$ be the maximal dimension, $r=\max_\tau r_\tau$ the maximal rank and $m$ the maximal order of the connection tensors. We assume that for every tree tensor network $X_{\bar\tau}$ we have that $Z_{\bar\tau}=F(t,X_{\bar\tau})$ is again a tree tensor network, with ranks $s_\tau \le cr$ for all subtrees $\tau\le\bar\tau$ with a moderate constant $c$. Then, one time step of the tree tensor integrator given by Algorithms~\ref{alg:ttn}--\ref{alg:Psi-tau} can be implemented such that it requires
\begin{itemize}
\item $O(dr(n+r^{m}))$ storage,
\item $O(ld)$ tensorizations/matricizations of matrices/tensors with $\le r^{m+1}$ entries,
\item $O(ld^2r^2(n+ r^{m}))$ arithmetical operations and
\item $O(d)$ evaluations of the function $F$,
\end{itemize}
provided the differential equations in Algorithms~\ref{alg:Phi-tau-i} and~\ref{alg:Psi-tau} are solved approximately using a fixed number of function evaluations per time step.
\end{lemma}

\begin{proof} The only nontrivial count regards the arithmetical operations. In Algorithm~\ref{alg:Phi-tau-i}, there are two QR decompositions of $r^m\times r$ matrices, which requires $O(r^m\cdot r^2)=O(r^{m+2})$ operations. In total over all subtrees of $\bar\tau$,  the total computational cost for the QR decompositions is thus $O(dr^{m+2})$.

The computation of $X_{\tau_i}^0 \times_0 \S_{\tau_i}^{0,\top}$ in factorized form requires only computing the smaller product
$C_{\tau_i}^0 \times_0 \S_{\tau_i}^{0,\top}$, where $C_{\tau_i}^0$ is the core tensor of $X_{\tau_i}^0$, as at the end of the previous subsection. The computational cost for this product is $O(r^{m+2})$ operations, and in total over all subtrees of $\bar\tau$,  the total computational cost is thus $O(dr^{m+2})$ operations.

The matrix product $ \U_{\tau_i}^{1,\top}\mat_0\bigl(F_{\tau_i}(t,X_{\tau_i}^1\times_0 {\S}_{\tau_i}^\top) \bigr)^\top$ is computed recursively using \eqref{UW}. This requires $O(|L(\tau_i)| (n r^2 +  r^{m+2}))$ operations, where $|L(\tau_i)|$ is the number of leaves of $\tau_i$. In total over all subtrees of $\bar\tau$, the computational cost for these products is therefore $O(ld (n r^2 +  r^{m+2}))$ operations.
%, where $l<d$ is the number of levels (i.e., the height of the tree $\bar\tau$.

To evaluate the function $F_{\tau_i}$, we need prolongations (which do not require any arithmetical operations) and restrictions. The computational cost for computing a restriction $\pi_{\tau,i}^\dagger$ in the way described in the previous subsection, is $O(|L(\tau_i)| (n r^2 +  r^{m+2}))$ operations for the recursive computation of the matrix $\mathbf{R}_{\tau_i}$ via \eqref{UW}. The mode-0 multiplication of
$\mathbf{R}_{\tau_i}$ with the core tensor requires another $O(r^{m+2})$ operations. One evaluation of $F_{\tau_i}$ requires several restrictions from the root down to $\tau_i$ and then costs
$O(d (n r^2 +  r^{m+2}))$ operations in addition to the evaluation of $F$. In total over all subtrees of $\bar\tau$,   the computational cost then becomes $O( l d^2  (n r^2 +  r^{m+2}))$ operations.

In Algorithm~\ref{alg:Psi-tau}, we note that the tree tensor network $Z_\tau = F_\tau(t,Y_\tau)$, when written in factorized form as
$
Z_\tau=G_\tau \times_0 \I_\tau \bigtimes_{i=1}^m \W_{\tau_i}
$
for a tree $\tau=(\tau_1,\dots\tau_m)$, has
$$
Z_\tau \bigtimes_{i=1}^m \U_{\tau_i}^\top = G_\tau \times_0 \I_\tau \bigtimes_{i=1}^m \U_{\tau_i}^\top \W_{\tau_i}.
$$
The products $\U_{\tau_i}^\top \W_{\tau_i}$ are computed recursively via \eqref{UW} in $O( |L(\tau)|  (n r^2 +  r^{m+2}))$ operations.
In total over all subtrees of $\bar\tau$,   the computational cost then again becomes $O( d l  (n r^2 +  r^{m+2}))$ operations.
\end{proof}
\ecl
 
\section{Exactness property of the TTN integrator}
\label{sec:exact}

We will show that under a non-degeneracy condition, the TTN integrator with the tree rank $(r_\tau)$ reproduces time-dependent tree tensor networks $A(t)$ with the same tree rank {\it exactly} at every time step when the integrator is applied with $F(t,Y)=\dot A(t)$ and exact initial value $Y^0=A(t_0)$. Such an exactness result is already known for the special cases of projector-splitting integrators for low-rank matrices \cite{LuO14}, tensor trains / matrix product states \cite{LuOV15}, and Tucker tensors \cite{LuVW18}. The latter result will now be used in a recursive way to prove the exactness property of the TTN integrator.

We first formulate the non-degeneracy condition. 
%Given 
%a tree tensor network $Y_{\bar\tau}^0$ for a tree $\bar \tau\in\mathcal{T}$ of  dimensions $(n_\ell)_{\ell\in L(\bar\tau)}$ and ranks $(r_\tau)_{\tau\in T(\bar\tau)}$, i.e. (see Lemma~\ref{lem:mf} for the notation),
%$$
%Y_{\bar\tau}^0 \in \mathcal{M}(\bar\tau, (n_\ell)_{\ell\in L(\bar\tau)}, (r_\tau)_{\tau\in T(\bar\tau)}),
%$$ 
%we consider the associated restriction maps $\pi_{\tau,i}^\dagger$ for every subtree $\tau \le \bar\tau$, as defined in Section~\ref{subsec:F-tau}. For an arbitrary tensor 
%$Z\in\mathcal{V}_{\bar\tau}$  we let $Z_{\bar\tau}=Z$ and for each subtree $\tau=(\tau_1,\dots,\tau_m)\le \bar\tau$ we set, recursively from the root to the leaves, 
%$$
%Z_{\tau_i} = \pi_{\tau,i}^\dagger (Z_\tau) \in \mathcal{V}_{\tau_i}.
%$$
%For a tree tensor network $Z\in  \mathcal{M}(\bar\tau, (n_\ell)_{\ell\in L(\bar\tau)}, (r_\tau)_{\tau\in T(\bar\tau)})$, the calculation in Section~\ref{subsec:phi-tau-dagger} shows that for each subtree $\tau\le \bar\tau$ we then have
%$$
%Z_{\tau} \in \mathcal{M}(\tau, (n_\ell)_{\ell\in L(\tau)}, (\widehat r_{\tau'})_{\tau'\in T(\tau)})
%\quad\text{with ranks} \  \widehat r_{\tau'} \le r_{\tau'}.
%$$
%Equality of the ranks holds if and only if all the matrices $\mathbf{R}_\tau$ as constructed in Section~\ref{subsec:phi-tau-dagger} are invertible.
Consider a time-dependent family of tree tensor networks $A(t)$ of full tree rank $(r_\tau)_{\tau\le\bar\tau}$,
%$$
%A(t) \in \mathcal{M}(\bar\tau, (n_\ell)_{\ell\in L(\bar\tau)}, (r_\tau)_{\tau\in T(\bar\tau)}), \qquad t_0\le t \le t^*_1,
%$$
and set $Y_{\bar\tau}^0 = A(t_0)$, for which we consider the restricted tensor networks $A_\tau(t):=\bigl(A(t)\bigr)_\tau$ defined by the restrictions \eqref{rest} associated with $Y_{\bar\tau}^0$  for the subtrees $\tau\le\bar\tau$. By Lemma~\ref{lem:rest-rank}, we then have for every subtree $\tau\le\bar\tau$ that
\begin{equation}\label{nondeg-0}
\text{$A_\tau(t_0)$ has full tree rank $(r_\sigma)_{\sigma\le\tau}$ for every subtree $\tau \le \bar\tau$.}
\end{equation}
We impose the condition that the same full-rank property still holds at $t_1>t_0$:
\begin{equation}\label{nondeg}
\text{$A_\tau(t_1)$ has full tree rank $(r_\sigma)_{\sigma\le\tau}$ for every subtree $\tau \le \bar\tau$.}
\end{equation}

\begin{theorem}[Exactness]
\label{thm:exact}
Let $A(t)$ be a continuously differentiable time-depen\-dent family of tree tensor networks $A(t)$ of full tree rank $(r_\tau)_{\tau\le\bar\tau}$  for ${t_0\le t \le t_1}$,
and suppose that the non-degeneracy condition \eqref{nondeg} is satisfied. Then the recursive TTN integrator used with the same tree rank $(r_\tau)_{\tau\in T(\bar\tau)}$ for $F(t,Y)=\dot A(t)$ is exact: starting from $Y^0 = A(t_0)$ we obtain  $Y^1 = A(t_1)$ .
\end{theorem}

\begin{proof} The result is obtained from the exactness result of the Tucker integrator that was proved in \cite{LuVW18} and an induction argument over the height of the trees. The height is defined in a formal way as follows:
	\begin{enumerate}[(i)]
		\item
			If $\tau = \ell \in \mathcal{L}$, then we set $ h(\tau) = 0 $; i.e., leaves have height 0.
		\item
			If $\tau = ( \tau_1, \dots, \tau_m) \in \mathcal{T} $, then 
			we set $ h(\tau) = 1 + \max \{ h(\tau_1), \dots, h(\tau_m) \} $.
	\end{enumerate}
We note that, since the restricitions $\pi_{\tau}^\dagger$ for $\tau\le \bar \tau$ do not depend on time $t$, time differentiation commutes with these linear maps and we have
$$
{\dot A}_\tau (t) := \frac{d}{d t} A_\tau(t) = \bigl( \dot A(t) \bigr)_\tau.
$$

(i) Consider first trees $\tau = ( \tau_1, \dots, \tau_m)$ of height 1. The tree tensor network $A_\tau(t)$ is then a Tucker tensor, which by \eqref{nondeg-0} and \eqref{nondeg} has  full multilinear rank 
$(r_\tau,r_{\tau_1},\dots,r_{\tau_m})$ at both $t=t_0$ and $t=t_1$. The TTN integrator with $F_\tau(t,Y_\tau)=\dot A_\tau(t)$ is in this case the same as the Tucker integrator of \cite{LuVW18} and hence reproduces $A_\tau(t_1)$ exactly by \cite[Theorem~4.1]{LuVW18}. 
\bcl We note that condition \eqref{nondeg} for the leaves $\tau_i$ corresponds to the invertibility condition in \cite[Theorem~4.1]{LuVW18}.
\ecl

(ii) For trees of height $k\ge 2$ we work with the induction hypothesis that the recursive TTN integrator with $F_\tau(t,Y_\tau)=\dot A_\tau(t)$ is exact for all trees $\tau < \bar\tau$ of height strictly smaller than $k$. For a tree $\tau=(\tau_1,\dots,\tau_m)$ of height $k$ the TTN integrator is therefore exact for the subtrees $\tau_i$, and hence the recursive steps in the TTN integrator are solved exactly. This reduces the recursive TTN integrator to the Tucker integrator for $F_\tau(t,Y_\tau)=\dot A_\tau(t)$, where by \eqref{nondeg-0} and \eqref{nondeg}, the tensor $A_\tau(t)$, viewed as a Tucker tensor $A_\tau(t)=C_\tau(t)\bigtimes_{i=1}^m \U_{\tau_i}(t)$, has  full multilinear rank $(r_\tau,r_{\tau_1},\dots,r_{\tau_m})$ at both $t=t_0$ and $t=t_1$. From the exactness result of \cite[Theorem~4.1]{LuVW18} it then follows that the integrator reproduces $A_\tau(t_1)$ exactly. This completes the induction argument. Finally, we thus obtain the exactness result for the maximal tree $\bar \tau$, which is the stated result.
\end{proof}

\section{Error bound}
\label{sec:err}

We derive an error bound for the integrator that is independent of singular values of matricizations of the connecting tensors, based on the corresponding result for Tucker tensors proved in \cite{LuVW18}, which in turn was based on the corresponding result for matrices proved in \cite{KiLW16}. We recall the notation $\mathcal{V}_\tau$ for the tensor space \eqref{V-tau} 
and $\M_\tau$ for the tree tensor network manifold \eqref{M-tau}. We set $\mathcal{V}=\mathcal{V}_{\bar\tau}$ and
$\mathcal{M}=\mathcal{M}_{\bar\tau}$ for the full tree $\bar\tau$.

We assume that $F:[0,t^*]\times \M \to \mathcal{V}$ is Lipschitz continuous and bounded,
\begin{align}
\label{L-bound}
 & \|F(t,Y) - F(t,\widetilde{Y})\| \le L\|Y-\widetilde{Y}\| & \text{ for all}\ Y,\widetilde{Y} \in \M,  
 \\[1mm]
 \label{B-bound}
 & \| F(t,Y) \| \le B & \text{ for all}\ Y \in \M. 
\end{align}\label{eq:F-bound}
Here and in the following, the chosen norm $\|\cdot\|$ is the tensor Euclidean norm. As usual in the numerical analysis of ordinary differential equations, this could be weakened to a local Lipschitz condition and local bound in a neighborhood of 
the exact solution $A(t)$ of the tensor differential equation \eqref{eq:fullEq} to the initial data $A(t_0)=A^0\in\mathcal{V}$, but for convenience we will work with the global Lipschitz condition and bound.

We further assume that $F(t,Y)$ is in the tangent space $\mathcal{T}_Y\M$ up to a small remainder: with $P(Y)$ denoting the orthogonal
projection onto $\mathcal{T}_Y\M$, we assume that for some $\eps>0$,
\begin{equation}\label{eps}
\| F(t,Y) - P(Y) F(t,Y) \| \le \eps 
\end{equation}
for all $(t,Y)\in [0,t^*]\times \M$ \bcl in some ball $\| Y \| \le \rho$, where it is assumed that the exact solution $A(t)$, $0\le t \le t^*$, has a bound that is strictly smaller than $\rho$.\ecl

Finally, we assume that the initial value $A^0$ and the starting value $Y^0\in\M$ of the numerical method are $\delta$-close:
\begin{equation} \label{init-err}
  \| Y^0-A^0 \| \le \delta.
\end{equation}

\begin{theorem}[Error bound]
  \label{thm:error}
  Under the above assumptions, the error of the numerical approximation $Y^n$ at $t_n=nh$, obtained with $n$ time steps of the TTN integrator with step size $h>0$, is bounded by
  \[
    \|Y^n - A(t_n)\| \le c_0\delta + c_1 \eps + c_2 h \qquad\text{for }\ t_n \le t^*,
  \]
where $c_i$ depend only on $L$, $B$, $t^*$, and the tree $\bar\tau$. \bcl This holds true provided that $\delta,\eps$ and $h$ are so small that the above error bound guarantees that  $ \|Y^n\| \le \rho$.\ecl
\end{theorem}

\bcl We remark that the error bound is still valid for sufficiently small $\delta,\eps$ and $h$ if the bound \eqref{eps} is satisfied only in some tubular neighborhood 
$\{ (t,Y) \in [0,t^*]\times \M \,:\, \| Y-A(t) \| \le \vartheta\}$ for an arbitrary fixed $\vartheta>0$. We do not include the proof of this more general result, because the higher technical intricacies would obscure the basic argument of the proof.
\ecl

The proof of Theorem~\ref{thm:error} works recursively, based on the corresponding result for Tucker tensors given in \cite{LuVW18} and using a similar induction argument to the proof of Theorem~\ref{thm:exact}. To make this feasible, we need that the conditions on $F=F_{\bar\tau}$ are also satisfied  for the reduced functions $F_\tau$ for every subtree $\tau\le \bar\tau$, 
which are constructed recursively in Definition~\ref{def:F-tau}. For $Y_\tau\in\M_\tau$, 
let $P_\tau(Y_\tau)$ be the orthogonal projection onto the tangent space $T_{Y_\tau}\M_\tau$.
We have the following remarkable property.

\begin{lemma} 
\label{lem:F-tau}
If $F_{\bar\tau}=F$ satisfies conditions \eqref{L-bound}--\eqref{eps}, then we have for every subtree $\tau\le\bar\tau$, with the same $L$, $B$, and $\eps$,
\begin{align}
\label{L-bound-tau}
 & \|F_\tau(t,Y_\tau) - F_\tau(t,\widetilde{Y_\tau})\| \le L\|Y_\tau-\widetilde{Y_\tau}\| & \text{ for all}\ Y_\tau,\widetilde{Y_\tau} \in \M_\tau,  
 \\[1mm]
 \label{B-bound-tau}
 & \| F_\tau(t,Y_\tau) \| \le B & \text{ for all}\ Y_\tau \in \M_\tau 
\end{align}
and 
\begin{equation}\label{eps-tau}
\| F_\tau(t,Y_\tau) - P_\tau(Y_\tau) F_\tau(t,Y_\tau) \| \le \eps 
\end{equation}
for all $(t,Y_\tau)\in [0,t^*]\times \M_\tau$ \bcl with $\| Y_\tau \| \le \rho$.\ecl
\end{lemma}

The proof of the $\eps$-bound \eqref{eps-tau} is based on the following lemma.

\begin{lemma}\label{tangential map}
\label{lem:M-tau}
Let $\tau=(\tau_1,\dots,\tau_m)\in\mathcal{T}$ and $i=1,\dots,m$. Let $M_\tau:\mathcal{M}_\tau \to \mathcal{V}_\tau$ be such that it maps into the tangent space:
$$
M_\tau(Y_\tau) \in T_{Y_\tau} \mathcal{M}_\tau  \quad\text{for all } \ Y_\tau\in  \mathcal{M}_\tau.
$$
Let $\pi_{\tau,i}^\dagger$ and $\pi_{\tau,i}$ be the restrictions and prolongations corresponding to some $Y_\tau^0\in \mathcal{M}_\tau$, and define
$$
M_{\tau_i} = \pi_{\tau,i}^\dagger \circ M_{\tau}  \circ \pi_{\tau,i}.
$$
Then, $M_{\tau_i}: \mathcal{M}_{\tau_i} \to \mathcal{V}_{\tau_i}$ also maps into the tangent space:
$$
M_{\tau_i}(Y_{\tau_i}) \in T_{Y_{\tau_i}} \mathcal{M}_{\tau_i}  \quad\text{for all } \ Y_{\tau_i}\in  \mathcal{M}_{\tau_i}.
$$
\end{lemma}

\begin{proof} Let $Y_{\tau_i}\in  \mathcal{M}_{\tau_i}$ and define $Y_\tau = \pi_{\tau,i}(Y_{\tau_i})$, which by Lemma~\ref{lem:prol-rank} is in $\mathcal{M}_\tau$. By assumption, $M_\tau(Y_\tau) \in T_{Y_\tau} \mathcal{M}_\tau $, and hence there exists a path 
$X_\tau(\theta)\in\mathcal{M}_\tau$, for $\theta$ near $0$, with 
$X_\tau(0)=Y_\tau$ and $\frac d{d\theta}\vert _{\theta=0} X_\tau = M_\tau(Y_\tau) $. Then, the restricted path
$X_{\tau_i}(\theta)=\pi_{\tau,i}^\dagger X_\tau(\theta)$ has
$$
X_{\tau_i}(0) = \pi_{\tau,i}^\dagger (Y_\tau) =  \pi_{\tau,i}^\dagger (\pi_{\tau,i}(Y_{\tau_i})) = Y_{\tau_i},
$$
because $\pi_{\tau,i}^\dagger$ is a left inverse of $\pi_{\tau,i}$ by Lemma~\ref{lem:prol-rest}. By local continuity of the rank, we then have
$X_{\tau_i}(\theta) \in \mathcal{M}_{\tau_i}$. Moreover,
$$
\frac d{d\theta}\bigg\vert _{\theta=0} X_{\tau_i} = \pi_{\tau,i}^\dagger\, \frac d{d\theta}\bigg\vert _{\theta=0} X_\tau = \pi_{\tau,i}^\dagger M_\tau(Y_\tau) =
M_{\tau_i}(Y_{\tau_i}).
$$
Hence, $M_{\tau_i}(Y_{\tau_i}) \in T_{Y_{\tau_i}} \mathcal{M}_{\tau_i}$.
\end{proof}

\begin{proof} (of Lemma~\ref{lem:F-tau})
The Lipschitz bound and the norm bound of $F_\tau$ follow directly from the corresponding bounds of $F$, using that restriction and prolongation are operators of norm 1 by Lemma~\ref{lem:prol-rest}. It then remains to show \eqref{eps-tau}.
For $Y\in\M$ we write
$$
F(t,Y) = P(Y)F(t,Y) + (I-P(Y))F(t,Y) \equiv M(t,Y) + R(t,Y)
$$
with $M(t,Y) \in T_Y\M$ and $\| R(t,Y) \| \le \eps$ by \eqref{eps}.
We let $M_{\bar\tau}=M$ and $R_{\bar\tau}=R$ and define recursively, for $\tau=(\tau_1,\dots,\tau_m)\le\bar\tau$,
\begin{align*}
M_{\tau_i} &= \pi_{\tau,i}^\dagger \circ M_{\tau} \circ \pi_{\tau,i},
\\
R_{\tau_i} &= \pi_{\tau,i}^\dagger \circ R_{\tau} \circ \pi_{\tau,i}.
\end{align*}
By Lemma~\ref{lem:M-tau}, for every subtree $\tau\le\bar\tau$, $M_\tau$ maps into the tangent space:
$$
M_{\tau}(Y_{\tau}) \in T_{Y_{\tau}} \mathcal{M}_{\tau}  \quad\text{for all } \ Y_{\tau}\in  \mathcal{M}_{\tau}.
$$
Hence,
$$
(I-P_\tau(Y_\tau))F_\tau(t,Y_\tau) = (I-P_\tau(Y_\tau))R_\tau(t,Y_\tau),
$$
and once again, since restriction and prolongation are operators of norm 1, it follows from \eqref{eps} that  
$$
\| (I-P_\tau(Y_\tau))F_\tau(t,Y_\tau) \| \le \| R_\tau(t,Y_\tau) \| \le \eps.
$$
This proves Lemma~\ref{lem:F-tau}.
\end{proof}

\begin{proof} (of Theorem~\ref{thm:error})
It suffices to assume that $Y^0=A(t_0) \in \M$, since the difference of exact solutions of the differential equation \eqref{eq:fullEq} corresponding to initial values that differ at most by $\delta$, is bounded by $c_0\delta$ for $t_0\le t \le t^*$ under the imposed Lipschitz condition on $F$.
Moreover, it then suffices to show that the local error after one time step is of magnitude $O(h(\eps+h))$. The result for the global error is then obtained with the familiar Lady Windermere's fan argument, as in \cite{KiLW16} and \cite{LuVW18}.

	As in the proof of Theorem~\ref{thm:exact}, we proceed by induction on the height of the tree.
	
	For trees of height 1, the recursive TTN integrator coincides with the Tucker integrator of \cite{LuVW18}, for which the error estimate has been proved in \cite{LuVW18}. 
	
	For trees $\tau=(\tau_1,\dots,\tau_m)$ of higher height, we observe that in the recursive TTN integrator, the differential equations for $Y_{\tau_i}$ are solved approximately by intermediate tree tensor networks with lower height, for which the $O(h(\eps+h))$ error bound holds by the induction hypothesis. If instead the differential equations for $Y_{\tau_i}$ were solved exactly, then the integrator would again reduce to the Tucker integrator and error in this idealized $Y_\tau$ after one step would be $O(h(\eps+h))$.
By studying the influence of the inexact solution of the differential equation for $Y_{\tau_i}$  on the error (as in \cite[Subsection 2.6.3]{KiLW16}), we find that the error of the actual $Y_\tau$ is still of magnitude $O(h(\eps+h))$. We omit the details of this perturbation argument, since it is cumbersome to write down explicitly and requires no ideas beyond using the triangle inequality.
	
	This completes the induction argument. Finally, we thus obtain the error bound for $\bar \tau$, which yields the result of Theorem~\ref{thm:error}.
\end{proof}

\section{Numerical experiments}
\label{sec:num}
\bcl
The recursive TTN integrator has already been applied to problems from plasma physics  \cite{LuE18} and quantum physics \cite{KlLR20}, where numerical results are reported.
%The numerical integrator presented in this paper generalizes and provides a systematic theoretical approach to the results presented in \cite{LuE18, LuOV15, LuVW18}. 
In the following we therefore give just two illustrative numerical examples. We choose the tree $\bar\tau$ of Figure 2.1. % \ref{fig:tree}.
%\begin{center}
%	\includegraphics[scale=0.25]{TensorTree.eps}
%\end{center}
%\begin{center}
%	\begin{tikzpicture}[every node/.style={circle,draw},level 1/.style={sibling distance=22mm},level 2/.style={sibling distance=10mm}]
%	\node[circle,draw, fill=black] { }
%	child { node[circle,draw, fill=black] {  }
%		child { node[circle,draw] { } }
%		child { node[circle,draw] { } }
%		child { node[circle,draw] { } } } 
%	child { node[circle,draw, fill=black] {  }
%		child { node[circle,draw] { } }
%		child { node[circle,draw] { } } }
%	child{ node[circle,draw] {}};
%	\end{tikzpicture}
%\end{center}
%The above picture represents a six-dimensional tensor. The white nodes represent the leaves $\U_i \in \R^{n_i \times r_i}$ and the intermediate dark nodes represent the connection tensors used in the definition of the tree tensor network.
The dimensions $n_\ell$ and the ranks $r_\tau$ are taken the same for all the nodes and are fixed to $n_\ell=n=16$ for all leaves $\ell$ and $r_\tau=r=5$ for all subtrees $\tau < \bar\tau$. We have chosen such small dimensions $n$ to be able to easily compute the reference solution, which is a full tensor with $n^6$ entries. In contrast, the storage for the tree tensor network for this tree is $6nr + r^4 + 2r^3$ entries.
\ecl
%The choice of higher parameters require optimization in the code that it is beyond the scope of the present work and they will be presented elsewhere. Suggestions and ideas for improving the implementation can be borrowed from \cite{KT14}.

\bcl 
The computations were done using Matlab R2017a software with Tensor Toolbox package v2.6 \cite{TTB_Software}. 
The implementation of the TTN is done using an Object Oriented Paradigm; we define in Matlab a class \texttt{Node} with two properties: \texttt{Value} and \texttt{Children}. The \texttt{Value} can be either a connection tensor or an orthonormal matrix. The \texttt{Children} is an ordered list of objects of type \texttt{Node}. If the \texttt{Children} list is empty, we are on one of the leaves of the tree; the \texttt{Value} of the \texttt{Node} is by construction an orthonormal matrix. The TTN is defined as an object of type \texttt{Node}.  %operations on TTNs are implemented via recursion process. 

The recursive TTN integrator is then implemented by recursively applying the Extended Tucker Integrator to an object of type \texttt{Node}. The recursion process in the recursive TTN algorithm is controlled by counting the elements of the \texttt{Children} list associated with the \texttt{Node}: if empty, we are on one of the leaves of the tree.

\subsection{Tree tensor network addition and retraction}
Let $\bar \tau$ be the given tree and let $ \mathcal{M}_{\bar \tau}$ be the manifold of tree tensor networks of given dimensions $(n_\ell)$ and tree rank $(r_\tau)_{\tau\le\bar\tau}$; see Section~\ref{subsec:ttn}.
 We consider the addition of  two given tensors $A \in \mathcal{M}_{\bar \tau}$ and $B \in \mathcal{T}_{A}\mathcal{M}_{\bar \tau}$ (a tangent tensor),
$$ C = A + B.$$
Then, $C $ is a tree tensor network on the same tree but of larger rank. We want to compute a tree tensor network retraction to the manifold $ \mathcal{M}_{\bar \tau}$, i.e., to the original tree rank $(r_\tau)_{\tau\le\bar\tau}$. 
\ecl
Such a retraction is typically required in optimization problems on low-rank manifolds and needs to be computed in each iterative step. The approach considered here consists of reformulating the addition
problem as the solution of the following differential equation at time $t = 1$:
$$ \dot{C}(t) = B, \quad C(0)=A.$$
We compare the approximation $Y^1 \in \mathcal{M}_{\bar \tau}$, computed with one time step of the recursive TTN integrator with step size $h=1$, with a different retraction, \bcl denoted by $X$, \ecl obtained by computing the full addition $C$ and recursively retracting to the manifold $\mathcal{M}_\tau$ for each $\tau \leq \bar \tau$. For the latter, we use the built-in function $\texttt{tucker\_als}$ of the Tensor Toolbox Package \cite{TTB_Software}; we recursively apply the function to the full tensor $C$ and its retensorized basis matrices.

 \bcl These comparisons are illustrated in Figure~\ref{fig:addition} where the norm of $B$ is varied. We observe both retractions $Y^1$ and $X$ have very similar error, and their difference is considerably smaller than their errors. Decreasing the norm of the tensor $B$ reduces the approximation error as expected, proportional to $\|B\|^2$.
This behaviour of the TTN integrator used for retraction is the same as observed for the Tucker integrator in \cite{LuVW18} for the analogous problem of the addition of a Tucker tensor of given multilinear rank and a tangent tensor.
\ecl

%For our numerical experiment, we consider a random TTN $A$ with tree, dimensions and ranks as indicated before; we take B as an element in the tangent space of the manifold $\mathcal{M}_{\bar \tau}$ at the point $A$. We compare the approximation $Y^1$ generated by the recursive TTN integrator with a retraction of the full solution, denoted by $X$. 

\begin{figure}[h!] \label{fig:addition}
	\begin{center}
		\includegraphics[width=\textwidth]{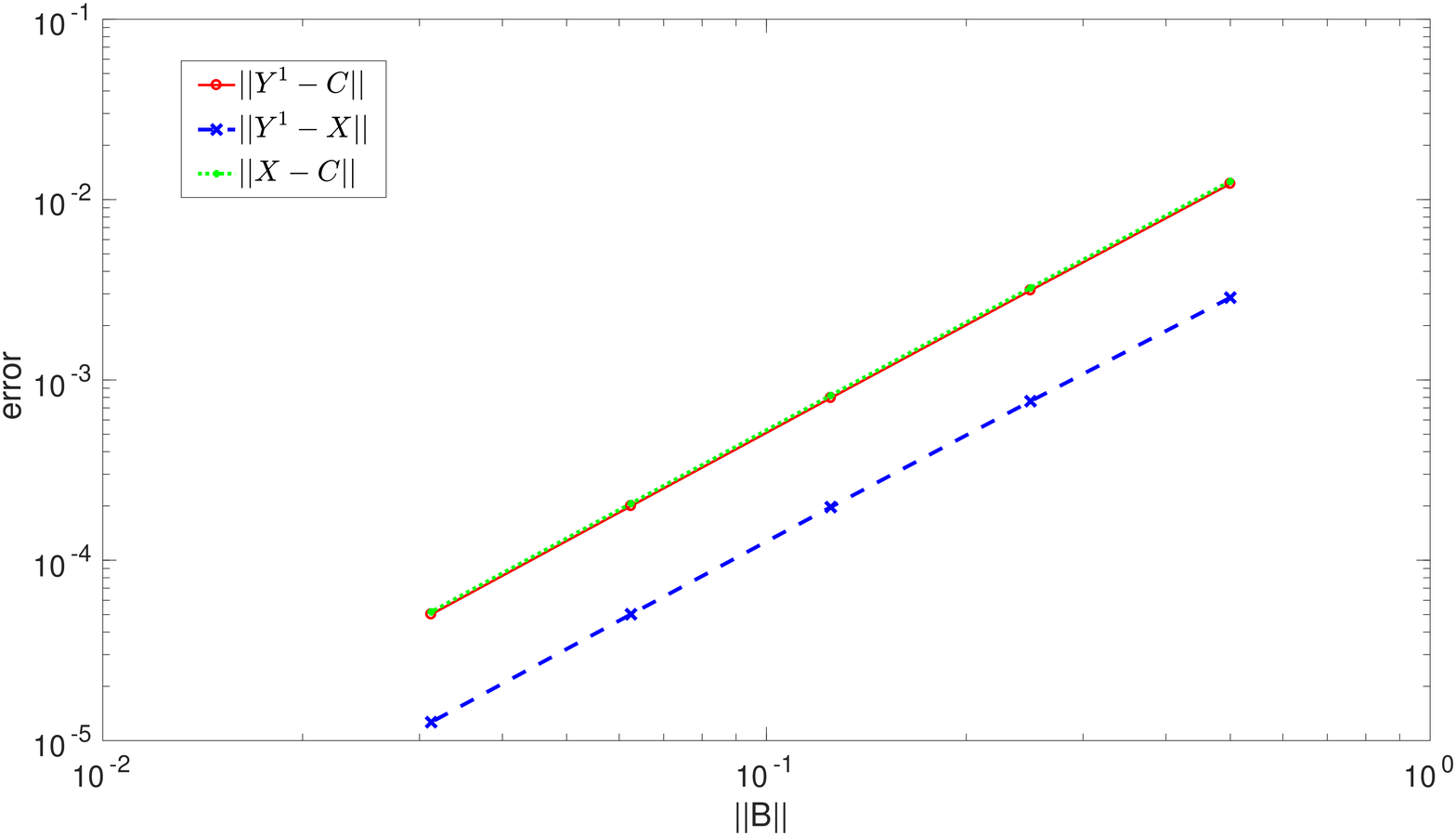}
		\caption{Error of retracted tree tensor network sum.}
	\end{center}
\end{figure}

The advantage of the retraction via the TTN integrator is that the result is completely built within the tree tensor network manifold. No further retraction is needed, which is favorable for storage and computational complexity.

\subsection{Verification of the exactness property}
We consider a tree tensor network $A^0 \in \mathcal{M}_{\bar \tau}$. For each subtree $\tau \leq \bar \tau$, let $\textbf{W}_{\tau} \in \mathbb{R}^{r_\tau \times r_\tau}$ be a skew-symmetric matrix which we choose of \bcl Frobenius \ecl norm 1. We consider a time-dependent tree tensor network $A(t) \in \mathcal{M}_{\bar \tau}$ such that $A(t_0)=A^0$ with basis matrices propagated in time through
$$ \U_\ell(t) = e^{t\textbf{W}_\ell} \U_\ell^0,\qquad \ell \in {L}(\bar \tau)  $$
and the connection tensors changed according to
%$$ \mat_0(C_\tau(t)) = e^{t\textbf{W}_\tau} \mat_0(C_\tau^0) .$$ 
$$ C_\tau(t) = C_\tau^0 \times_0 e^{t\textbf{W}_\tau}, \qquad \tau \leq \bar \tau, \ \tau \notin {L}(\bar \tau). $$
The time-dependent tree tensor network does not change rank and as predicted by Theorem~\ref{thm:exact}, it is reproduced exactly by the recursive TTN integrator, up to round-off errors. The absolute errors $\|Y_n - A(t_n) \|$ calculated at time $t_n=nh$ with step sizes $h= 0.1, \, 0.01, \, 0.001$ until time $t^*=1$ are shown in Figure~\ref{fig:Exactness}. 

\begin{figure}[ht]
	\includegraphics[width=\textwidth]{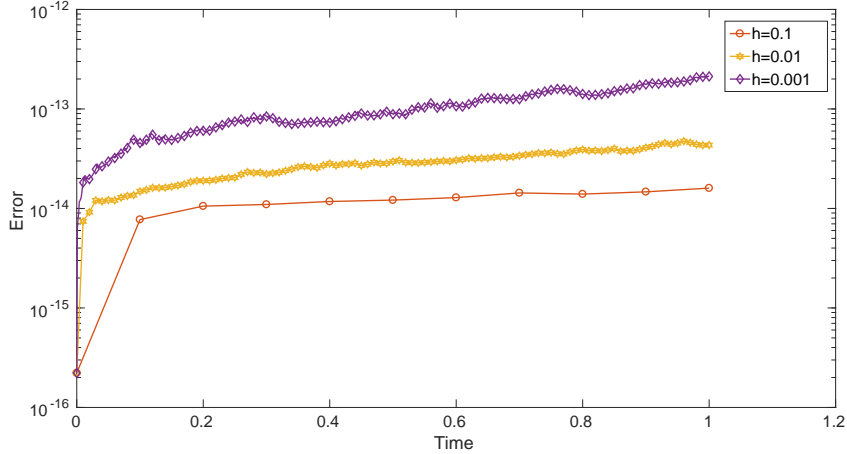}
	\caption{Error vs.~time in a case of exactness up to round-off errors.}
	\label{fig:Exactness}
\end{figure}

%\subsection{Exactness}
%In the second example, we consider a random tree tensor network $A^0$ with tree structure as indicated before. Let \textbf{W} be a skew-symmetric matrix of norm 1. We consider a time-dependent tree tensor network $A(t)$ such that $A(t_0)=A^0$ with basis matrices propagated in time through
%$$ \U_i(t) = e^{t\textbf{W}} \U_i^0 \ .$$ 
%The time dependent tree tensor network does not change rank and as predicted by Theorem~\ref{thm:exact}, we have exactness. The absolute errors $\|A(t_n) - Y(t_n) \|_F$ calculated at time $t_n=nh$ with $h=0.01$ until time $T=50$ are shown in Figure~\ref{fig:Exactness}. These errors are only caused by round-off errors.
%
%\begin{figure}[ht]
%	\includegraphics[width=\textwidth]{AbsoluteExactnessError.eps}
%	\caption{Exactness Absolute Errors computed until time $T=50$.}
%	\label{fig:Exactness}
%\end{figure}

\section*{Acknowledgements}
%This work was supported by Deutsche Forschungsgemeinschaft, GRK~1838 and SFB~1173. 
\bcl
We thank two anonymous referees for their helpful comments on a previous version.

This work was funded by the Deutsche Forschungsgemeinschaft (DFG, German Research Foundation) --- Project-ID 258734477 --- SFB 1173 and DFG GRK 1838.
\ecl

\bibliography{references}{}

\end{document}